\documentclass{amsart}
\usepackage{amscd}
\usepackage{epic}
\usepackage[all,2cell,dvips]{xy} \UseAllTwocells \SilentMatrices
\usepackage{stmaryrd}
\usepackage{amssymb}
\usepackage{setspace} 
\usepackage{mathdots}
\usepackage{ifsym}

\calclayout

\newtheorem{theorem}{Theorem}[section]
\newtheorem{conj}[theorem]{Conjecture}

\newtheorem{cor}[theorem]{Corollary}
\newtheorem{lem}[theorem]{Lemma}

\theoremstyle{definition}
\newtheorem{beisp}[theorem]{Example}
\newtheorem{definit}[theorem]{Definition}

\newtheorem*{rem*}{Remark}

\newfont{\xlarge}{cmbx10 scaled \magstep2}

\DeclareMathOperator{\Inn}{Inn}

\DeclareMathOperator{\aut}{Aut}

\DeclareMathOperator{\So}{SO}
\DeclareMathOperator{\Gl}{GL}
\DeclareMathOperator{\Sl}{SL}
\DeclareMathOperator{\Sp}{SP}
\DeclareMathOperator{\oo}{O}

\DeclareMathOperator{\Span}{Span}
\DeclareMathOperator{\M}{M}

\DeclareMathOperator{\Id}{Id}

\renewcommand{\epsilon}{\varepsilon}
\renewcommand{\phi}{\varphi}
\renewcommand{\rho}{\varrho}
\renewcommand{\theta}{\vartheta}

\renewcommand{\leq}{\leqslant}
\renewcommand{\le}{\leqslant}

\renewcommand{\ge}{\geqslant}

\newfont{\myfont}{msam10 scaled 2000}

\title[Isomorphy Classes of Involutions of $\Sp(2n, k)$, $n>2$]{\boldmath
Isomorphy Classes of Involutions of $\Sp(2n, k)$, $n>2$}
\author{Robert W. Benim}
\address {Department of Mathematics and Computer Science\\
Pacific University\\
Forest Grove, OR, 97116}
\email{rbenim@gmail.com}
\author{Aloysius G. Helminck}
\thanks{Third author is partially supported by N.S.F. Grant DMS-0532140}
\address{Department of Mathematics\\
North Carolina State University\\
Raleigh, N. C., 27695} \email{loek@math.ncsu.edu}
\author{Farrah Jackson}
\address {Department of Mathematics and Computer Science\\
Elizabeth City State University\\
Elizabeth City, N. C., 27909}
\email{fjchandler@mail.ecsu.edu}

\begin{document}
\maketitle

\begin{abstract}
A first characterization of the isomorphism classes of $k$-involutions for any reductive algebraic groups defined over a perfect field was given in \cite{Helm2000} using $3$ invariants.  In \cite{HWD04} a classification of all involutions on $\Sl(n,k)$ for $k$ algebraically closed, the real numbers, the $p$-adic numbers or a finite field was provided. In this paper, we build on these results to develop a detailed characterization of the involutions of $\Sp(2n,k)$. We use these results to classify the isomorphy classes of involutions of $\Sp(2n, k)$ where $k$ is any field not of characteristic 2.

\end{abstract}

\section{Introduction}
Let $G$ be a connected reductive algebraic group defined over a
field $k$ of characteristic not $2$, $\theta$ an involution of $G$
defined over $k$, $H$ a $k$-open subgroup of the fixed point group
of $\theta$ and $G_k$ (resp. $H_k$) the set of $k$-rational points
of $G$ (resp. $H$). The variety $G_k/H_k$ is called a symmetric
$k$-variety. For $k=\mathbb{R}$ these symmetric
$k$-varieties are also called real reductive symmetric spaces. 
These varieties occur in many problems in
representation theory, geometry and singularity theory. 
To study
these symmetric $k$-varieties one needs first a classification of
the related $k$-involutions. A characterization of the isomorphism
classes of the $k$-involutions was given in \cite{Helm2000}
essentially using the following 3 invariants:
\begin{enumerate}
\item classification of admissible $(\Gamma, \theta)$-indices.
\item classification of the $G_k$-isomorphism classes of
$k$-involutions of
 the $k$-an\-iso\-tro\-pic kernel of $G$.
\item classification of the $G_k$-isomorphism classes of $k$-inner
elements of $G$.
\end{enumerate}
For more details, see \cite{Helm2000}. The admissible $(\Gamma,
\theta)$-indices determine most of the fine structure of the
symmetric $k$-varieties and a classification of these was included
in \cite{Helm2000} as well. For $k$ algebraically closed or $k$ the
real numbers the full classification can be found in \cite{Helm88}.
For other fields a full classification of the remaining two invariants is
still lacking. In particular the case of symmetric $k$-varieties
over the $p$-adic numbers is of interest. We note that
the above characterization was only proven for $k$ a perfect field.

In \cite{HWD04} a full characterization of the isomorphism
classes of $k$-involutions was given in the case that $G=\Sl(n, k)$  which 
does not depend on any of the results in
\cite{Helm2000}. It was also shown how one may construct an outer-involution from a given non-degenerate symmetric or skew-symmetric bilinear form $\beta$ of $k^n$.  Using this characterization the
possible isomorphism classes for $k$ algebraically
closed, the real numbers, the $p$-adic numbers and finite
fields were classified.

In this paper we build upon the results of \cite{HWD04} to give a characterization 
of involutions of $\Sp(2n, k)$, the symplectic group.

We first show that if an automorphism $\theta= \Inn_A$ where $A \in \Gl (2n,\overline{k})$ leaves $\Sp(2n,k)$ invariant, then 
we can assume $A$ in $\Sp(n,k[\sqrt{\alpha}])$ where $k[\sqrt{\alpha}]$ is a quadratic extension of $k$. 
Thus, to classify the involutions of $\Sp(2n,k)$ it suffices to determine which $A \in \Sp(2n,k[\sqrt{\alpha}])$ induce involutions of $\Sp(2n,k)$, and to then determine the isomorphy classes of these involutions over $\Sp(2n,k)$. Using these results, we give a full classification of involutions of $\Sp(2n,k)$ for $k$ algebraically closed, the real numbers, or a finite field.

\section{Preliminaries}
Our basic reference for
reductive groups will be the papers of Borel and Tits
\cite{Borel-Tits65}, \cite{Borel-Tits72} and also the books of
Borel \cite{Borel91},  Humphreys \cite{Humph75} and Springer
\cite{Spring81}. We shall follow their notations and terminology.
All algebraic groups and algebraic varieties are  taken over an
arbitrary field $k$ (of characteristic $\neq 2$) and all algebraic
groups considered are linear algebraic groups.

Our main reference for results regarding involutions of $\Sl(n,k)$ will be \cite{HWD04}.  Let
$k$ be a field of characteristic not $2$, $\bar k$ the algebraic closure of $k$,
$$\M(n,k)=\{ n\times n \text{-matrices with
entries in $k$} \}, $$
$$\Gl(n,k)= \{ A\in \M(n,k)\mid \det
(A)\neq 0\}$$  and
$$\Sl(n,k)= \{ A\in \M(n,k)\mid \det
(A)=1\}. $$  Let $k^*$ denote the product group of all the nonzero field
elements, $(k^*)^2=\{a^2\mid a\in k^*\}$ and $I_n \in \M(n,k)$
denote the identity matrix. We will sometimes use $I$ instead of $I_n$ when the dimension of the identity matrix is clear.

We recall some important definitions and theorems from \cite{HWD04}.
\begin{definit}
\label{isoinv} Let $G$ be an algebraic group defined over a field $k$, and let $G_k$ be the set of $k$-rational points. Let ${\aut(G_k)}$ denote the set of all automorphisms
of $G_k$. For $A\in \Gl(n,k)$ let $\Inn_{A}$ denote the inner
automorphism defined by $\Inn_{A}(X)=A^{-1}XA$ for all ${X\in
\Gl(n,k)}$. Let ${\Inn_k(G_k)}=\{\Inn_A\mid A\in{G_k}\}$ denote the set
of all \textit{inner} automorphisms of $G_k$ and let ${\Inn(G_k)}$
denote the set of automorphisms $\Inn_{A}$ of $G_k$ with $A\in
G$ such that $\Inn_{A}(G_k)=G_k$. If $\Inn_A$ is order 2, that is $\Inn_A^2$ is the identity but $\Inn_A$ is not, then we call $\Inn_A$ an \textit{inner involution} of $G_k$. 
We say that $\theta$ and $\tau$ in ${\aut(G_k)}$ are
\textit{$\Inn(G_k)$-isomorphic} if there is a $\phi$ in ${\Inn(G_k)}$
such that $\tau=\phi ^{-1}\theta\phi$. Equivalently,  we say that $\tau$ and $\theta$ are in the same \textit{isomorphy class}.

\end{definit}

In \cite{HWD04}, the isomorphy classes of the inner-involutions of $\Sl(n,k)$ were classified, and they are as follows:

\begin{theorem} \label{sltheorem1}
Suppose the involution $\theta \in \aut(\Sl(n,k))$ is of inner type. Then
up to isomorphism $\theta$ is one of the following:
\begin{enumerate}
\item \label{sltheorem1.1} $\Inn_Y|_G$, where $Y = I_{n-i,i} \in
\Gl(n,k)$ where $i \in \left\{1, 2, \dots, \lfloor  \frac
{n}{2}\rfloor \right\}$ where $$I_{n-i,i} =  \left(\begin{array}{cc}I_{n-i} & 0 \\0 & I_i\end{array}\right)$$. 

\item \label{sltheorem1.2} $\Inn_Y|_G$,
where
$Y =  L_{\frac n 2, x} \in \Gl(n,k)$ where $x \in k^*\slash
k^{*2}$, $x \not\equiv 1\mod k^{*2}$ and
$$L\sb{n,x}=\begin{pmatrix} 0 & 1 & \hdots & 0 & 0 \\
x & 0 & \hdots & 0 & 0
\\ \vdots & \vdots & \ddots &
\vdots & \vdots \\ 0 & 0 & \hdots & 0 & 1\\ 0 & 0 & \hdots & x & 0
\end{pmatrix} .$$
 \end{enumerate}
Note that $(ii)$ can only occur when $n$ is even.
\end{theorem}

For the purposes of this paper, we will use matrices of the form $\left(\begin{smallmatrix}0 & I_{\frac{n}{2}} \\xI_{\frac{n}{2}} & 0\end{smallmatrix}\right)$ (and there multiples) rather than $L_{\frac n 2, x}$. Either of these serves as a member of the isomorphy class listed in the previous theorem. We will eventually see that all of the isomorphy classes of $\Sp(2n,k)$ are just isomorphy classes of $\Sl(n,k)$ that have been divided into multiple isomorphy classes.

We now begin to define the notion of a symplectic group. To do this, we must first define orthogonal groups. Let $M$ be the matrix of a non-degenerate bilinear form $\beta$
over $k^n$ 
with respect to a basis $\{ e_1, \dots e_n \}$ of $V$.
We will say that $M$ is the matrix of $\beta$ if 
the basis $\{ e_1, \dots e_n \}$ 
is the standard basis of $k^n$. 

The typical notation for the orthogonal group is $\oo(n,k)$, which is the group $$\oo(n,k)= \{ A\in \M(n,k)\mid (Ax)^T(Ay) = x^Ty\}.$$ This group consists of the matrices which fix the standard dot product. This can be generalized to any non-degenerate bilinear $\beta$, which will yield the group $$\oo(n,k,\beta)= \{ A\in \M(n,k)\mid \beta(Ax,Ay) = \beta(x,y) \}.$$ If $M$ is the matrix of $\beta$ with respect to the standard basis, then we can equivalently say $$\oo(n,k,\beta)= \{ A\in \M(n,k)\mid A^TMA = M \}.$$ It is clear from this definition that all matrices in $\oo(n,k,\beta)$ have determinant 1 or $-1$. We are interested in the case where $M$ is a skew-symmetric matrix. 

We note a couple of important facts, the first of which will be used repeatedly throughout this paper.
\begin{enumerate}

\item Skew-symmetric matrices of even dimension are congruent to the matrix $J = J_{2n} = \left(\begin{smallmatrix}0 & I_n \\-I_n & 0\end{smallmatrix}\right).$

\item If $\beta_1$ and $\beta_2$ correspond to $M_1$ and $M_2$, then $\oo(n,k,\beta_1)$ and $\oo(n,k,\beta_2)$ are isomorphic via 
$$\Phi:\oo(n,k,\beta_1) \rightarrow \oo(n,k,\beta_2): X \rightarrow Q^{-1}XQ$$ for some $Q \in \Gl(n,k)$ if $Q^TM_1Q = M_2$ ($M_1$ and $M_2$ are congruent via $Q$).

\end{enumerate}

So, we will assume that $\beta$ is such that we can replace $M$ with the matrix $J$. When we do this, then we write $\Sp(2n,k) = \oo(n,k,\beta)$, and we call this the \textit{ Symplectic Group}. It can be shown that all matrices in $\Sp(2n,k)$ have determinant 1, so in fact $\Sp(2n,k)$ is a subgroup of $\Sl(2n,k)$. Lastly, note that to classify the involutions of an orthogonal group where $M$ is skew-symmetric, one can apply the classification that will follow by simply using the isomorphism given above.

\section{Automorphisms of $\Sp(2n,k)$}

It follows from a proposition on page 191 of \cite{Borel91} that $\aut(\Sp(2n,\overline{k})) /\Inn(\Sp(2n,\overline{k})))$ must be a subgroup of the diagram automorphisms of the Dynkin diagram $C_n$. Since $C_n$ only has the trivial diagram autormphism, then we have that $\aut(\Sp(2n,\overline{k})) = \Inn(\Sp(2n,\overline{k}))$. When $k$ is not algebraically closed, then all automorphisms of $\Sp(2n,k)$ will still be of the form $\Inn_A$ for some $A \in \Sp(n,\overline{k})$ since all automorphisms of $\Sp(2n,k)$ must also be an automorphism of $\Sp(n,\overline{k}).$ Thus, the classifications and characterizations that follow in this paper consider all automorphisms and involutions of $\Sp(2n,k)$.

We now examine which automorphisms will act as the identity on $\Sp(2n,k)$. This will prove to be useful when we classify matrix representatives for automorphisms.

\begin{theorem}\label{Inn=Id}
Let $G = \Sp(2n,k)$. If $\Inn_{A}|_G = \Id$ for some $A \in \Gl(2n,\overline{k})$  then $A=pI$ for some $p\in{\overline{k}}.$
\end{theorem}

\begin{proof}
Suppose $\Inn_{A}|_G=\Id$ for some $A\in{\Gl(2n,\overline{k})}$.  Then for all $X\in{G}$ we have $\Inn_{A}(X)=A^{-1}XA=X$ which means that  $AX=XA$ for all $X\in{G} $.
 Let 
 $$A=\begin{pmatrix} A_1 & A_2 \\ A_3 & A_4 \end{pmatrix} $$ and consider the matrix
$$W_1=\begin{pmatrix} I_n & I_n\\ 0 & I_n \end{pmatrix}. $$\\
Since $W_1\in{G}, AW_1=W_1A$ which implies\\
\begin{align*}
\begin{pmatrix} A_1 & A_2 \\ A_3 & A_4 \end{pmatrix}\begin{pmatrix} I_n & I_n \\ 0 & I_n \end{pmatrix}&=\begin{pmatrix} I_n & I_n \\ 0 & I_n \end{pmatrix}\begin{pmatrix} A_1 & A_2 \\ A_3 & A_4 \end{pmatrix}\\
\begin{pmatrix} A_1 & A_1+A_2 \\ A_3 & A_3+A_4 \end{pmatrix}&=\begin{pmatrix} A_1 +A_3 & A_2 +A_4\\ A_3 & A_4 \end{pmatrix}.
\end{align*}

Hence,  $A_3=0$ and $A_1=A_4$.  With this information in hand we are now able to rewrite $A$ as  $A=\begin{pmatrix} A_1 & A_2 \\ 0 & A_1 \end{pmatrix}$. We now consider the matrix  $W_2=\begin{pmatrix} I_n & 0 \\ I_n & I_n \end{pmatrix}.$  Now $W_2$ is also in $G$ and thus $AW_2=W_2A$ and thus
\begin{align*}
\begin{pmatrix} A_1 & A_2 \\ 0 & A_1 \end{pmatrix}\begin{pmatrix} I_n & 0 \\ I_n & I_n \end{pmatrix}&=\begin{pmatrix} I_n & 0 \\ I_n & I_n \end{pmatrix}\begin{pmatrix} A_1 & A_2 \\ 0 & A_1 \end{pmatrix}\\
\begin{pmatrix}  A_1+A_2 & A_2 \\ A_1 & A_1 \end{pmatrix}&=\begin{pmatrix} A_1  & A_2\\ A_1 & A_2+A_1 \end{pmatrix}.
\end{align*}
Which implies that $A_2 =0$ and thus $ A=\begin{pmatrix} A_1 & 0 \\ 0 & A_1 \end{pmatrix}$. \\
Let $$\bar{X_k}= \begin{pmatrix}X_k & 0\\ 0 & X_k \end{pmatrix} $$ where
$$X_{k}:=\begin{pmatrix}I_{n-k-1}
& \hdots & 0 \\ \vdots & -1 & \vdots \\
0 & \hdots & I_k
\end{pmatrix}$$
and $k=0,1,...,n-1.$  Then $\bar{X_k}\in G$ and hence we may utilize the fact that $A\bar{X}_k=\bar{X}_kA$, to conclude that 

$$\begin{pmatrix}A_1X_k & 0\\ 0 & A_1X_k \end{pmatrix} = \begin{pmatrix}X_kA_1 & 0 \\ 0 & X_kA_1 \end{pmatrix}.$$

From the above equality we see that  $A_1X_k = X_kA_1$.   Define $A_1=(a_{i,j})$ for $i,j=1,2,..,n$.  Then $A_1X_k = X_kA_1$ implies 

$$\begin{pmatrix} a_{11} & a_{12} & \hdots & -a_{1,n-k} & \hdots & a_{1,n}\\
a_{21} & a_{22} & \hdots & -a_{2,n-k} & \hdots & a_{2,n}\\
\vdots & \vdots & \hdots & \vdots & \hdots & \vdots \\
a_{n-k,1} & a_{n-k,2} & \hdots & -a_{n-k,n-k} & \hdots & a_{n-k,n}\\
\vdots & \vdots & \hdots & \vdots & \hdots & \vdots \\
a_{n,1} & a_{n,2} & \hdots & -a_{n,n-k} & \hdots & a_{n,n}
\end{pmatrix} =$$
$$\begin{pmatrix} a_{11} & a_{12} & \hdots & a_{1,n-k} & \hdots & a_{1,n}\\
a_{21} & a_{22} & \hdots & a_{2,n-k} & \hdots & a_{2,n}\\
\vdots & \vdots & \hdots & \vdots & \hdots & \vdots \\
-a_{n-k,1} & -a_{n-k,2} & \hdots & -a_{n-k,n-k} & \hdots & -a_{n-k,n}\\
\vdots & \vdots & \hdots & \vdots & \hdots & \vdots \\
a_{n,1} & a_{n,2} & \hdots & a_{n,n-k} & \hdots & a_{n,n} \end{pmatrix}.$$

\noindent Hence, it follows that  $ a_{n-k,j} = a_{j,n-k} =0$ for  $j\not= n-k$ and $k=0,1...,n-1, \; \; j=1,2,..,n$. Therefore we now obtain the fact that  $A$ is a diagonal matrix say, 
$$A= \begin{pmatrix} A_d & 0 \\ 0 & A_d

\end{pmatrix}  \; \; \text{with} \; \; A_d=\begin{pmatrix} a_{11} & 0 & \hdots & 0\\  0 & a_{22} & \hdots & 0\\
\vdots & \vdots & \ddots & 0\\ 0 & 0 & \hdots & a_{n,n} \end{pmatrix}.$$

Let $$\overline{Y_l} = \begin{pmatrix} Y_l & 0\\ 0 & Y_l \end {pmatrix} \; \; \text{where} \; \; \;
Y_l= \left(\begin{array}{ccccc}I_l & 0 & 0 & \hdots & 0 \\0 & 0 & 1 & \hdots & 0 \\0 & 1 & 0 & \hdots & 0 \\\vdots & \vdots &  &  & I_{{n-l-2}\times {n-l-2}} \\0 & 0 &  &  & \end{array}\right)$$

\noindent and $l=0,1,..., n-2$.  Then $\overline{Y_l}\in{\Sp(2n,k)}$ and again $A\overline{Y_l}=\overline{Y_l}A$ which implies $A_dY_l=Y_lA_d$.  Therefore, we obtain the following equality

$$\begin{pmatrix} a_{11} & 0 & 0 & 0 & 0 & 0 & 0 & 0 & 0\\ 0 & a_{22} & 0 & 0 & 0 & 0 & 0 & 0 & 0\\
0 & 0 & \ddots & 0 & 0 & 0 & 0 & 0\\ 0 & 0 & 0 & a_{ll} & 0 & 0 & 0 & 0 & 0\\
0 & 0 & 0 & 0 & 0 & a_{l+1,l+1} & 0 & 0  \\ 0 & 0 & 0 & 0 & a_{l+2,l+2} & 0 & 0 & 0  & 0 \\
0 & 0 & 0 & 0 &  0 & 0 &a_{l+3,l+3} & 0 & 0 \\  0 & 0 & 0 & 0 & 0 & 0 & 0 & \ddots & 0\\
 0 & 0 & 0 & 0 & 0 & 0 & 0 & 0 & a_{n,n} \end{pmatrix} =$$
 
$$ \begin{pmatrix} a_{11} & 0 & 0 & 0 & 0 & 0 & 0 & 0 & 0\\ 0 & a_{22} & 0 & 0 & 0 & 0 & 0 & 0 & 0\\
0 & 0 & \ddots & 0 & 0 & 0 & 0 & 0\\ 0 & 0 & 0 & a_{ll} & 0 & 0 & 0 & 0 & 0\\
0 & 0 & 0 & 0 & 0 & a_{l+2,l+2} & 0 & 0  \\ 0 & 0 & 0 & 0 & a_{l+1,l+1} & 0 & 0 & 0  & 0 \\
0 & 0 & 0 & 0 &  0 & 0 &a_{l+3,l+3} & 0 & 0 \\  0 & 0 & 0 & 0 & 0 & 0 & 0 & \ddots & 0\\
 0 & 0 & 0 & 0 & 0 & 0 & 0 & 0 & a_{n,n} \end{pmatrix}$$

Hence $a_{l+1,l+1} = a_{l+2, l+2}$ for $l=0,1,...,n-2$. That is $A=p\Id$ for some $p\in{\overline{k}}$.
\end{proof}

The following is a list of notation which will be used in the proof of Theorem \ref{charthms}.

\medskip

\noindent \quad Let $X_{r.s}$ be the $n\times n$ diagonal matrix with a $-1$ in the $(r,r)$ and $(s,s)$ entries and $1$'s everywhere else.

\medskip

\noindent \quad Let $X_r$ be the $n\times n$ diagonal matrix with a $-1$ in the $(r,r)$ position and $1$'s everywhere else.

\medskip
\noindent  \quad Let $E_{r,s}$ be the $n\times n$ matrix with a $1$ in the $(r,s)$ entry and $0's$ everywhere else.

\medskip
\noindent \quad Let $T_c$ be the $c\times c$ antidiagonal matrix with $1$'s on the antidiagonal and $0$'s everywhere else.

\medskip
\noindent \quad Let $I_c$ be the $c\times c$ identity matrix.  If the size of the identity matrix is understood from the context then $I$ may be used to represent $I_c$.

\begin{align}
J_{2n}&=J=\left(\begin{array}{cc}0 & I_n \\-I_n & 0\end{array}\right)\\
Y_{r,s}&=\left(\begin{array}{cc}T_{r+s-1} & 0 \\0 & I_{n-(r+s-1)}\end{array}\right)\\
Z_{r,s}&=\begin{pmatrix} Y_{r,s} & 0\\E_{r,s} &
Y_{r,s}\end{pmatrix}\\
\bar{Z_{r,s}}&=\begin{pmatrix} -Y_{r,s} & 0\\E_{r,s} &
-Y_{r,s}\end{pmatrix}\\
Z'_{r,s}&=\begin{pmatrix}
Y_{r-n,s-n} & E_{r-n,s-n}\\0 & Y_{r-n,s-n}\end{pmatrix}\\
\bar{Z}'_{r,s}&=\begin{pmatrix} Y_{r-n,s-n} & E_{r-n,s-n}\\0 &
Y_{r-n,s-n}\end{pmatrix}\\
M_{r,s}&=\begin{pmatrix} E_{s-n,r} &
Y_{s-n,r}\\-Y_{s-n,r} & 0\end{pmatrix}\\
\bar{M}_{r,s}&=\begin{pmatrix} E_{s-n,r} &
-Y_{s-n,r}\\Y_{s-n,r} & 0\end{pmatrix}\\
U_{r,s}&=\begin{pmatrix} I_{-n+(r+s-1)} & 0\\0 &
T_{2n-(r+s-1)}\end{pmatrix}\\
 V_{r,s}&=\begin{pmatrix} U_{r,s} & 0\\E_{r,s} &
U_{r,s}\end{pmatrix}\\
\bar{V}_{r,s}&=\begin{pmatrix} -U_{r,s} & 0\\E_{r,s} &
-U_{r,s}\end{pmatrix}\\
V'_{r,s}&=\begin{pmatrix} U_{r-n,s-n} & E_{r-n,s-n}\\0 &
U_{r-n,s-n}\end{pmatrix}\\
\bar{V'}_{r,s}&=\begin{pmatrix} -U_{r-n,s-n} & E_{r-n,s-n}\\0 &
-U_{r-n,s-n}\end{pmatrix}\\
N_{r,s}&=\begin{pmatrix} E_{s-n,r} &
U_{s-n,r}\\-U_{s-n,r} & 0\end{pmatrix}\\
\bar{N_{r,s}}&=\begin{pmatrix} E_{s-n,r} &
-U_{s-n,r}\\U_{s-n,r} & 0\end{pmatrix}\\
W_{r,s}&=\begin{pmatrix} T_n & 0\\E_{r,s} & T_n\end{pmatrix}\\
\bar{W}_{r,s}&= \begin{pmatrix} -T_n & 0\\E_{r,s} &
-T_n\end{pmatrix}\\
W'_{r,s}&=\begin{pmatrix} T_n &
E_{r-n,s-n}\\0 & T_n\end{pmatrix}\\
\bar{W}'_{r,s}&=\begin{pmatrix} -T_n & E_{r-n,s-n}\\0 &
-T_n\end{pmatrix}\\
F_{r,s}&=\begin{pmatrix}
E_{s-n,r} &T_n\\T_n & 0\end{pmatrix}\\
\end{align}

\begin{align}
\bar{F}_{r,s}&=\begin{pmatrix} E_{s-n,r} & -T_n\\-T_n & 0\end{pmatrix}
\end{align}

%%%%%%%%%%%%%%%%%%%%%%%%%%%%%%%%%%%%
% Characterization Theorem from Farrah Jackson

We now have the following result that characterizes inner-automorphisms of $\Sp(2n,k)$. We will see that for $\Inn_A$ to be an inner-involution of $\Sp(2n,k)$, that we can not only assume that $A$ is symplectic, but for the entries of $A$, we do not need the algebraic closure of the field $k$, but either the field itself or a quadratic extension of $k$.

\begin{theorem}
\label{charthms}
Suppose $A\in\Gl(2n,\bar{k})$ , $\bar{G}=\Sp(2n,\bar{k})$ and $G=\Sp(2n,k)$ .

\begin{enumerate}

\item The inner automorphism $\Inn_A$ keeps $\Sp(2n,\bar{k})$ invariant if and only if $A=pM$ for some
$p\in\bar{k}$ and $M\in \Sp(2n,\bar{k})$.

\item  If $A\in \Sp(2n,\bar{k})$, then $\Inn_A$ keeps $\Sp(2n,k)$ invariant if and only if we can show $A \in \Sp(2n, k(\sqrt{\alpha}))$ where each entry of $A$ is a $k$-multiple of $\sqrt{\alpha},$ for some $\alpha \in k$.  

\end{enumerate}\
\end{theorem}

\begin{proof}
\begin{enumerate}

\item $\Longleftarrow$  Suppose $A=pM$ for some  $p\in\bar{k}$  and
$M\in\bar{G}$. Let $X\in\bar{G}$, then
$$\Inn_A(X) = \Inn_{pM}(X)=(pM)^{-1}X(PM)=M^{-1}XM$$
Since $M, M^{-1} , X \in\bar{G},  M^{-1}XM\in\bar{G}$ and thus $\Inn_A$ keeps $\bar{G}$ invariant.\\

$\Longrightarrow$ Suppose $\Inn_A$ keeps $\bar{G}$ invariant.  Then for
any $X\in\bar{G}, $

 $B=\Inn_A(X)=A^{-1}XA\in\bar{G}$.   Since $B\in\bar{G}$, by definition $B^TJB=J$ which implies that
  $B=J^{-1}(B^T)^{-1}J$.  In addition, since $B=A^{-1}XA$, we have that $ (B^T)^{-1}=A^T(X^T)^{-1}(A^T)^{-1}$.  Thus the following is true
 $$A^{-1}XA=B$$ implies
 $$A^{-1}XA=J^{-1}(B^T)^{-1}J$$ which implies
 $$A^{-1}XA=J^{-1}(A^T(X^T)^{-1}(A^T)^{-1})J$$ hence 
 $$X=AJ^{-1}A^T(X^T)^{-1}(A^T)^{-1}JA^{-1}.$$

 Now since $X\in\bar{G}$, we know $(X^T)^{-1}= JXJ^{-1}$ which means
 $$X=AJ^{-1}A^T(JXJ^{-1})(A^T)^{-1}JA^{-1}$$ that is 
$$ X=(AJ^{-1}A^TJ)X(AJ^{-1}A^TJ)^{-1}$$
$$\text{i.e.} \; \; \Inn_{AJ^{-1}A^TJ}(X)=X. $$
Therefore by Lemma \ref{Inn=Id} $AJ^{-1}A^TJ=q\Id$ for some $q\in\bar{k}^*$ which implies
$q^{-1}A J^{-1}A^TJ =\Id$.  Let $p\in\bar{k}^*$ such that $p^2=q^{-1}$.  Then for $M=pA$ we have
$$MJ^{-1}M^TJ=pAJ^{-1}pA^TJ=p^2AJ^{-1}A^{T}J=q^{-1}AJ^{-1}A^TJ=I.$$
Therefore, $MJ^{-1}M^TJ=\Id$ which implies $M^TJM=J$ ie. $M\in\bar{G}.$\\

\item $\Longleftarrow$  Suppose $A=pM$ for some $p\in\bar{k}$ and
$M\in G$.  Let $X\in G$, then
$$\Inn_A(X)=\Inn_{pM}(X)=p^{-1}M^{-1}XpM=M^{-1}XM$$
Since $M^{-1}, X, M \in G$ we know $\Inn_A(X)=M^{-1}XM\in G$
and thus $\Inn_A$ keeps $G$ invariant.

\medskip

$\Longrightarrow$ Suppose $A=(a_{ij})\in\bar{G}$ and $\Inn_A$ keeps $G$
invariant.

We will first show that $a_{ri}a_{rj}+a_{si}a_{sj}\in G.$\\

\noindent \textbf{CASE 1:}  Suppose $r,s\leq n$.\\
\noindent \textbf{Subcase a:}  Suppose $i\leq n$.\\
The $(i,j)$ entry of $\Inn_A(J)$ is given by
$$a_{1,n+i}a_{1,j}+a_{2,n+i}a_{2,j}+...+a_{2n,n+i}a_{2n,j}\in k$$
since $J\in G$ and $\Inn_A$ keeps G invariant.  By the same argument
the $(i,j)$ entry of
 $\Inn_A\begin{pmatrix} I & I\\0 & I\end{pmatrix}$ given by\\
$a_{1j}a_{n+1,n+i}+a_{2j}a_{n+2,n+i}+...+a_{nj}a_{2n,n+i}+a_{n+1,n+i}a_{n+1,j}+
a_{n+2,n+i}a_{n+2,j}+...+a_{2n,n+i}a_{2n,j}-a_{1,n+i}a_{n+1,j}-a_{2,n+i}a_{n+2,j}-
...-a_{n,n+i}a_{2n,j}\in k$\\
Hence the $(i,j)$ position of $\Inn_A(J) -\Inn_A\begin{pmatrix} I & I\\
0 & I\end{pmatrix}$  given by\\
$-a_{1j}a_{n+1,n+i}-a_{2j}a_{n+2,n+i}-...-a_{nj}a_{2n,n+i}+a_{1,n+i}a_{1,j}+a_{2,n+i}a_{2,j}
+...+a_{n,n+i}a_{n,j}+a_{1,n+i}a_{n+1,j}+a_{2,n+i}a_{n+2,j}+
...+a_{n,n+i}a_{2n,j}\in k.$\\
We know the matrix $\begin{pmatrix} I & 0\\ X_{rs} & I
\end{pmatrix}$ is in G and hence the $(i,j)$ entry of $\Inn_A\begin{pmatrix} I & 0\\ X_{rs} & I
\end{pmatrix}$ given by \\
$a_{1j}a_{n+1,n+i}+a_{2j}a_{n+2,n+i}+...+a_{nj}a_{2n,n+i}-a_{1,n+i}a_{1,j}-a_{2,n+i}a_{2,j}
-(-a_{r,n+i}a_{r,j})-a_{r+1,n+i}a_{r+1,j}-...-(-a_{s,n+i}a_{s,j})-a_{s+1,n+i}a_{s+1,j}
-...-a_{n,n+i}a_{n,j}-a_{1,n+i}a_{n+1,j}-a_{2,n+i}a_{n+2,j}-...-a_{2n,n+i}a_{2n,j}\in k$\\
Now, the $(i,j)$ entry of $\Inn_A(J) -\Inn_A\begin{pmatrix} I & I\\
0 & I\end{pmatrix}+ \Inn_A\begin{pmatrix} I & 0\\ X_{rs} & I
\end{pmatrix}$ is given by $2a_{r,n+i}a_{rj}+2a_{s,n+i}a_{s,j}$ and
hence $a_{rl}a_{rj}+a_{sl}a_{sj}\in k$ for all $l>n $ and\\
$j=1,2,...,2n$.\\

\noindent \textbf{Subcase b:}  Suppose $i>n$\\
For $i>n$ the $(i,j)$ entry of $\Inn_A(J)$ yields\\
$-a_{1,i-n}a_{1,j}-a_{2,i-n}a_{2,j}-...-a_{2n,i-n}a_{2n,j}$\\
and the $(i,j)$ position of $\Inn_A\begin{pmatrix} I & I\\
0 & I\end{pmatrix}$ for $i>n$ is \\
$-a_{n+1,i-n}a_{1j}-a_{n+2,i-n}a_{2,j}-...-a_{2n,i-n}a_{n,j}-a_{n+1,i-n}a_{n+1,j}
-a_{n+2,i-n}a_{n+2,j}-...-a_{2n,i-n}a_{2n,j}+a_{1,i-n}a_{n+1,j}+a_{2,i-n}a_{n+2,j}
+...+a_{n,i-n}a_{2n,j}.$\\
Hence the $(i,j)$ entry of $\Inn_A(J) -\Inn_A\begin{pmatrix} I & I\\
0 & I\end{pmatrix}$ is given by \\
$a_{n+1,i-n}a_{1j}+a_{n+2,i-n}a_{2,j}+...+a_{2n,i-n}a_{n,j}
-a_{1,i-n}a_{1,j}-a_{2,i-n}a_{2,j}-...-a_{n,i-n}a_{n,j}
-a_{1,i-n}a{n+1,j}-a_{2,i-n}a_{n+2,j} -...-a_{n,i-n}a_{2n,j}.$\\
For $i>n$ the $(i,j)$ entry of $\Inn_A\begin{pmatrix} I & 0\\ X_{rs} & I\end{pmatrix}$ is\\

%%%%%%%%%%%%%%%%%%%%%%%%%%%%%%
%look at this

%%%%%%%%%%%%%%%%%%%%%%%%%%%%%
$-a_{n+1,i-n}a_{1,j}-a_{n+2,i-n}a_{2j}-...-a_{2n,i-n}a_{n,j} +
a_{1,i-n}a_{1,j}+a_{2,i-n}a_{2,j}+...+a_{n,i-n}a_{n,j}+
a_{1,i-n}a_{n+1,j}+a_{2,i-n}a_{n+2,j}+...+a_{n,i-n}a_{2n,j}.$\\

%%%%%%%%%%%%%%%%%%%%%%%%%%%%%%%%%%%%%%%%%

Therefore the $(i,j)$ entry of $\Inn_A(J) -\Inn_A\begin{pmatrix} I & I\\
0 & I\end{pmatrix}+ \Inn_A\begin{pmatrix} I & 0\\ X_{rs} & I
\end{pmatrix}$ yields $-2a_{r,i-n}a_{rj}-2a_{s,i-n}a_{s,j}$ and
since $i>n$ we have that $a_{rl}a_{rj}+a_{sl}a_{sj}\in k$ for all $l\leq n$ and $j=1,2,...,2n$.  Combining subcases a and b we have that
$a_{rl}a_{rj}+a_{sl}a_{sj}\in k$ whenever $r,s\leq n$.\\

\noindent \textbf{CASE 2:}  Suppose $r,s > n$.  Without loss of generality assume
$r<s.$\\
\noindent \textbf{Subcase a:}  Suppose $i\leq n$.  Now the matrix $\begin{pmatrix} I & 0\\
I & I\end{pmatrix}$ is in $G$ and since $\Inn_A$ keeps $G$ invariant
the $(i,j)$ entry of $\Inn_A\begin{pmatrix} I & 0\\
I & I\end{pmatrix}$ given by\\
$a_{1,j}a_{n+1,n+i}+a_{2,j}a_{n+2,n+i}+...+a_{n,j}a_{2n,n+i}-a_{1,n+i}a_{1,j}
-a_{2,n+i}a_{2,j}-...-a_{n,n+i}a_{n,j}-a_{1,n+i}a_{n+1,j}-a_{2,n+i}a_{n+2,j}
-...-a_{n,n+i}a_{2n,j} \in k$\\
Now the $(i,j)$ entry of $\Inn_A(J)$ was given in case 1 subcase a,
therefore the $(i,j)$ entry of $\Inn_A(J) +\Inn_A\begin{pmatrix} I & 0\\
I & I\end{pmatrix}$ is\\
$a_{1,j}a_{n+1,n+i}+a_{2,j}a_{n+2,n+i}+...+a_{n,j}a_{2n,n+i} +
a_{n+1,n+i}a_{n+1,j}+a_{n+2,n+i}a_{n+2,j}+...+a_{2n,n+i}a_{2n,j}-
a_{1,n+i}a_{n+1,j}-a_{2,n+i}a_{n+2,j}-...-a_{n,n+i}a_{2n,j}$\\
which must lie in $k$.  We know the matrix $\begin{pmatrix} I & X_{r-n,s-n}\\
0 & I\end{pmatrix}\in G$ and thus the automorphism $\Inn_A\begin{pmatrix} I & X_{r-n,s-n}\\
0 & I\end{pmatrix}\in G$ and its $(i,j)$ entry given by\\
$a_{1,j}a_{n+1,n+i}+a_{2j}a_{n+2,n+i}+...+a_{n,j}a_{2n,n+i}+
a_{n+1,n+i}a_{n+1,j}+a_{n+2,n+i}a_{n+2,j}+...+(-a_{r,n+i}a_{r,j})+a_{r+1,n+i}a_{r+1,j}
+...+(-a_{s,n+i}a_{s,j})+a_{s+1,n+i}a_{s+1,j}+...+a_{2n,n+i}a_{2n,j}-a_{1,n+i}a_{n+1,j}
-a_{2,n+i}a_{n+2,j}-...-a_n{n+i}a_{2n,j} \in k.$\\
Finally we observe that the $(i,j)$ entry of $\Inn_A(J) +\Inn_A\begin{pmatrix} I & 0\\
I & I\end{pmatrix}-\\ \Inn_A\begin{pmatrix} I & X_{r-n,s-n}\\
0 & I\end{pmatrix}$ is given by
$2a_{r,n+i}a_{rj}+2a_{s,n+i}a_{s,j}$ and
hence $a_{rl}a_{rj}+a_{sl}a_{sj}\in k$ for all $l>n $ and
$j=1,2,...,2n$.\\

\noindent \textbf{Subcase b:} Suppose $i>n$.  The $(i,j)$ entry of $\Inn_A\begin{pmatrix} I & 0\\
I & I\end{pmatrix}$ is in $k$ and is given by\\
$-a_{n+i,i-n}a_{1,j}-a_{n+2,i-n}a_{2,j}-...-a_{2n,i-n}a_{n,j}+a_{1,i-n}a_{1,j}+
a_{2,i-n}a_{2,j}+...+a_{n,i-n}a_{n,j}+a_{1,i-n}a_{n+1,j}+a_{2,i-n}a_{n+2,j}+...+
a_{n,i-n}a_{2n,j}.$\\
Hence the $(i,j)$ position of $\Inn_A(J) + \Inn_A\begin{pmatrix} I & 0\\
I & I\end{pmatrix}$ is\\
$-a_{n+i,i-n}a_{1,j}-a_{n+2,i-n}a_{2,j}-...-a_{2n,i-n}a_{n,j}-a_{n+1,i-n}a_{n+1,j}
-a_{n+2,i-n}a_{n+2,j}-...-a_{2n,i-n}a_{2,j}+a_{1,i-n}a_{n+1,j}+a_{2,i-n}a_{n+2,j}+...+
a_{n,i-n}a_{2n,j}$\\ must reside in $k$.  For $i>n$ the $(i,j)$ entry of $\Inn_A\begin{pmatrix} I & X_{r-n,s-n}\\
0 & I\end{pmatrix}$ is given by\\
$-a_{n+1,i-n}a_{1,j}-a_{n+2,i-n}a_{2,j}-...-a_{2n,i-n}a_{n,j}-a_{n+1,i-n}a_{n+1,j}-a_{n+2,i-n}a_{n+2,j}
-...-(-a_{r,i-n}a_{rj})-a_{r+1,i-n}a_{r+1,j}-...-(-a_{s,i-n}a_{s,j})-a_{s+1,i-n}a_{s+1,j}-...-
a_{2n,i-n}a_{2n,j}+a_{1,i-n}a_{n+1,j}+a_{2,i-n}a_{n+2,j}+...+a_{n,i-n}a_{2n,j}.$\\
Therefore by considering the $(i,j)$ entry of $\Inn_A(J) + \Inn_A\begin{pmatrix} I & 0\\
I & I\end{pmatrix}- \\ \Inn_A\begin{pmatrix} I & X_{r-n,s-n}\\
0 & I\end{pmatrix}$ we see that $-2a_{r,i-n}a_{rj}-2a_{s,i-n}a_{s,j}$ 
must be in $k$.  Since we assumed $i>n$
we have that $a_{rl}a_{rj}+a_{sl}a_{sj}\in k$ for all $l\leq n$ and
$j=1,2,...,2n$.  By combining subcases a and b we obtain
$a_{rl}a_{rj}+a_{sl}a_{sj}\in k$ whenever $r,s> n.$\\

\medskip

\noindent \textbf{CASE 3:}  Suppose $r\leq n$ and $s>n.$

\textbf{Subcase a: } Suppose $i\leq n$.  The matrix $\begin{pmatrix} I & 0\\
X_r & I\end{pmatrix}\in G$ and therefore $\Inn_A\begin{pmatrix} I & 0\\
X_r & I\end{pmatrix}\in G$. Specifically, the $(i,j)$ entry of $\Inn_A\begin{pmatrix}
 I & 0\\ X_r & I\end{pmatrix}$ given by\\
 $a_{1,j}a_{n+1,n+i}+a_{2,j}a_{n+2,n+i}+...+a_{n,j}a_{2n,n+i}-a_{1,n+i}a_{1,j}
 -a_{2,n+i}a_{2,j}-...-(-a_{r,n+i}a_{r,j})-a_{r+1,n+i}a_{r+1,j}-...-a_{n,n+i}a_{n,j}
 -a_{1,n+i}a_{n+1,j}-a_{2,n+i}a_{n+2,j}-...-a_{n,n+i}a_{2n,j}$\\
 
lies in $k$.  Now the $(i,j)$ entry of $\Inn_A(J)+\Inn_A\begin{pmatrix} I & 0\\
X_r & I\end{pmatrix}$, which must be in $k$, is\\
$a_{1,j}a_{n+1,n+i}+a_{2,j}a_{n+2,n+i}+...+a_{n,j}a_{2n,n+i}+2a_{r,n+i}a_{r,j}
+a_{n+1,n+i}a_{n+1,j}+\\
a_{n+2,n+i}a_{n+2,j}+...+a_{2n,n+i}a_{2n,j}-a_{1,n+i}a_{n+1,j}
-a_{2,n+i}a_{n+2,j}-...-a_{n,n+i}a_{2n,j}.$\\
If we now consider the automorphism $\Inn_A$ on the  matrix
$\begin{pmatrix} I & X_{s-n}\\ 0 & I\end{pmatrix}\in G$ then we
see that the $(i,j)$ entry of $\Inn_A\begin{pmatrix} I & X_{s-n}\\ 0
& I\end{pmatrix}$ is given by\\
$a_{1,j}a_{n+1,n+i}+a_{2,j}a_{n+2,n+i}+...+a_{n,j}a_{2n,n+i}
+a_{n+1,n+i}a_{n+1,j}+a_{n+2,n+i}a_{n+2,j}+...+(-a_{s,n+i}a_{s,j})+...+a_{2n,n+i}a_{2n,j}-a_{1,n+i}a_{n+1,j}
-a_{2,n+i}a_{n+2,j}-...-a_{n,n+i}a_{2n,j}.$\\
Hence, the $(i,j)$ entry of  $\Inn_A(J)+\Inn_A\begin{pmatrix} I & 0\\
X_r & I\end{pmatrix}-\Inn_A\begin{pmatrix} I & X_{s-n}\\ 0 &
I\end{pmatrix}$ gives us $2a_{r,n+i}a_{r,j}+2a_{s,n+i}a_{s,j}$ and
more importantly since we assumed $i\leq n$ we have that
$a_{r,l}a_{r,j}+a_{s,l}a_{s,j}\in k$ for all $l>n$ and
$j=1,2,...,2n$.\\

\noindent \textbf{Subcase b: } Suppose $i>n$.  For $i>n$ the $(i,j)$ entry of
$\Inn_A\begin{pmatrix}I & 0\\ X_r & I\end{pmatrix}$ yields\\
$-a_{n+1,i-n}a_{i,j}-a_{n+2,i-n}a_{2,j}-...-a_{2n,i-n}a_{n,j}+a_{1,i-n}a_{1,j}
+a_{2,i-n}a_{2,j}+...+)-a_{r,i-n}a_{r,j}+...+a_{n,i-n}a_{n,j}+a_{1,i-n}a_{n+1,j}+
a_{2,i-n}a_{n+2,j}+...+a_{n,i-n}a_{2n,j}.$\\
 Therefore the $(i,j)$ entry of $\Inn_A(J) +\Inn_A\begin{pmatrix}I & 0\\ X_r & I\end{pmatrix}$
is\\
$-a_{n+1,i-n}a_{i,j}-a_{n+2,i-n}a_{2,j}-...-a_{2n,i-n}a_{n,j}-2a_{r,i-n}a_{r,j}
-a_{n+1,i-n}a_{n+1,j}-a_{n+2,i-n}a_{n+2,j}-...-a_{2n,i-n}a_{2n,j}
a_{1,i-n}a_{n+1,j}+a_{2,i-n}a_{n+2,j}+...+a_{n,i-n}a_{2n,j.}$\\
Lastly we consider the $(i,j)$ entry of $\Inn_A\begin{pmatrix} I &
X_{s-n}\\ 0 & I\end{pmatrix}$ which is given by\\
$-a_{n+1,i-n}a_{i,j}-a_{n+2,i-n}a_{2,j}-...-a_{2n,i-n}a_{n,j}
-a_{n+1,i-n}a_{n+1,j}-a_{n+2,i-n}a_{n+2,j}-...-(-a_{s,i-n}a_{s,j})-...-a_{2n,i-n}a_{2n,j}
a_{1,i-n}a_{n+1,j}+a_{2,i-n}a_{n+2,j}+...+a_{n,i-n}a_{2n,j}.$\\
So the $(i,j)$ entry of $\Inn_A(J)+\Inn_A\begin{pmatrix} I & 0\\
X_r & I\end{pmatrix}-\Inn_A\begin{pmatrix} I & X_{s-n}\\ 0 &
I\end{pmatrix}$ gives us $-2a_{r,i-n}a_{r,j}-2a_{s,i-n}a_{s,j}$
and since $i>n$ we have that $a_{r,l}a_{ar,j}+a_{s,l}a_{s,j}\in k$
for all $l\leq n$ and $j=1,2,...,2n.$\\ Combining subcases a and b we
have that
$a_{rl}a_{rj}+a_{sl}a_{sj}\in k$ whenever $r\leq n$ and $s>n$\\

\noindent In conclusion, by combining Cases 1,2,and 3 we can conclude that
 $a_{r,i}a_{r,j}+a_{s,i}a_{s,j}\in k$ for all $i,j=1,2,...2n$ and $r\neq s$.\\

%%%%%%%%%%%%%%%%%%%%%%%%%%%%%%%%%%%%%%%%%%%%%
%AriArj in k

%%%%%%%%%%%%%%%%%%%%%%%%%%%%%%%%%%%%%%%%%%%%%

We are now able to use the fact that
$a_{r,i}a_{r,j}+a_{s,i}a_{s,j}\in k$ for all $i,j=1,2,...2n$ and
$r\neq s$ to show that $a_{r,i}a_{r,j}\in k$ for all
$i,j=1,2,...,2n$.  However, we must show this in two cases. We
will first show that $a_{r,l}a_{r,j}\in k$ for all $l\leq n$ and
then show that $a_{r,l}a_{r,j}\in k$ for all $l>n$. Without loss
of generality it shall suffice to show
$a_{1,l}a_{1,j}\in k$ for all $l$. \\

\noindent \textbf{CASE 1:} Assume $i>n$. The $(i,j)$ entry or $\Inn_A(J)$  is given by\\
$-a_{1,i-n}a_{1,j}-a_{2,i-n}a_{2,j}-...-a_{2n,i-n}a_{2n,j}$ which
is in $k$ and implies that\\
$a_{1,i-n}a_{1,j}+a_{2,i-n}a_{2,j}+...+a_{2n,i-n}a_{2n,j}\in k.$
From our previous argument we know that
$a_{r,i}a_{r,j}+a_{s,i}a_{s,j}\in k$ for all $i,j=1,2,...,2n$, so
obviously $a_{r,i}a_{r,j}+a_{s,i}a_{s,j}\in k$ for $i>n$.  Making
use of that fact the equality given by
$$a_{1,i-n}a_{1,j}=$$
$(a_{1,i-n}a_{1,j}+a_{2,i-n}a_{2,j}+...+a_{2n,i-n}a_{2n,j})
-(1/2)(a_{2,i-n}a_{2,j}+a_{3,i-n}a_{3,j})-\\(1/2)(a_{3,i-n}a_{3,j}+
a_{4,i-n}a_{4,j})-(1/2)(a_{4,i-n}a_{4,j}+a_{5,i-n}a_{5,j})
-...-(1/2)(a_{2n,i-n}a_{2n,j}+a_{2,i-n}a_{2,j})$\\
must be in $k$, ie. $a_{1,i-n}a_{1,j}\in k$.  Since we assumed
that $i>n$ we have that $a_{1,l}a_{1,j}\in k$ for $l\leq n$.
Furthermore, we can conclude that $a_{r,l}a_{r,j}\in k$ for $l\leq
n$\\

\noindent \textbf{CASE 2:}  Assume $i\leq n$.  Then the $(i,j)$ entry of $\Inn_A(J)$,
which is in $k$, is given by
$a_{1,i+n}a_{1,j}+a_{2,i+n}a_{2,j}+...+a_{2n,i+n}a_{2n,j}.$
We again make use of the fact that
$a_{r,i}a_{r,j}+a_{s,i}a_{s,j}\in k$ for $i=1,2,...,2n$, and have
an equality similar to the one in case 1 ($i-n$ is simply replaced
by $i+n$)
$$a_{1,i+n}a_{1,j}=$$
$(a_{1,i+n}a_{1,j}+a_{2,i+n}a_{2,j}+...+a_{2n,i+n}a_{2n,j})
-(1/2)(a_{2,i+n}a_{2,j}+a_{3,i+n}a_{3,j})-\\(1/2)(a_{3,i+n}a_{3,j}+
a_{4,i+n}a_{4,j})-(1/2)(a_{4,i+n}a_{4,j}+a_{5,i+n}a_{5,j})
-...-(1/2)(a_{2n,i+n}a_{2n,j}+a_{2,i+n}a_{2,j})$\\
which again must be in $k$.  Since we assumed $i\leq n$ we have
that $a_{1,l}a_{1,j}\in k$ for $l> n$ and furthermore,
$a_{r,l}a_{r,j}\in k$ for $l>n$.  Combining cases 1 and 2 shows
that $a_{r,l}a_{r,j}\in k$ for $i,j=1,2,...,2n.$\\

\end{enumerate}

%%%%%%%%%%%%%%%%%%%%%%%%%%%%%%%%%%%%%%%%%%%%%%%%%%%%%%%%%

%PART 2 OF THEOREM

%%%%%%%%%%%%%%%%%%%%%%%%%%%%%%%%%%%%%%%%%%%%%%%%%%%%%%%%%

I will finally show that $a_{ri}a_{sj}\in k$ for $r\neq s$ \\

\noindent \textbf{CASE I:} Suppose $r,s\leq n$.  Without loss of generality we will
assume that $r<s$.\\

\begin{enumerate}

\item \textbf{Subcase 1}: Suppose $r+s<n+1$.   \\ Let
$Y_{r,s}=\begin{pmatrix} T_ {s+r-1} & 0\\0 &
I_{ n-(s+r-1)}\end{pmatrix}$ and\\
$$Z_{r,s}=\begin{pmatrix} Y_{r,s} & 0\\E_{r,s} &
Y_{r,s}\end{pmatrix}$$ 
 Now $Z_{r,s}\in
G$ and hence $\Inn_A$ must keep $Z_{r,s}$ invariant and thus all the
entries of $\Inn_A(Z_{r,s})$ must lie in$ k$.\\
\begin{enumerate}
\item Assume $i\leq n$.  Then the $(i,j)$ entry of $\Inn_A(Z_{r,s})$ is
given by\\
$-a_{r,n+i}a_{s,j}+a_{n+s+r-1,n+i}a_{1,j}+a_{n+s+r-2,n+i}a_{2,j}+...+\\
a_{n+2,n+i}a_{s+r-2,j}+a_{n+1,n+i}a_{s+r-1,j}-a_{1,n+i}a_{n+s+r-1,j}-
a_{2,n+i}a_{n+s+r-2,j}-...-a_{s+r-2,n+i}a_{n+2,j}-a_{s+r-1,n+i}a_{n+1,j}
+a_{n+r+s,n+i}a_{r+s,j}+a_{n+r+s+1,n+i}a_{r+s+1,j}+...+a_{2n,n+i}a_{n,j}
-a_{r+s,n+i}a_{n+r+s,j}-a_{r+s+1,n+i}a_{n+r+s+1,j}-...-a_{n,n+i}a_{2n,j}.$\\
Let $\bar{Z}_{r,s}\begin{pmatrix} -Y_{r,s} & 0\\E_{r,s} &
-Y_{r,s}\end{pmatrix}$. Now $\bar{Z}_{r,s}\in G$ and thus
$\Inn_A({\bar{Z}}_{r,s})\in G$.  In fact, the $(i,j)$ entry of
$\Inn_A(\bar{Z}_{r,s})$ is the negative of the $(i,j)$ entry of
$\Inn_A{Z}_{r,s}$ with the exception of $-a_{r,n+i}a_{s,j}$ which
remains negative.  Therefore, $\Inn_A({Z_{r,s}})+ \Inn_A(\bar{Z}_{r,s})$
has an $(i,j)$ entry of $-2a_{r,n+i}a_{s,j}$.  Since both
$\Inn_A({Z_{r,s}})$ and $\Inn_A(\bar{Z}_{r,s})$ are both in $G$ their sum
is in $G$ and hence $-2a_{r,n+i}a_{s,j}\in k$.  Since we assumed
$i\leq n$ we have $a_{r,l}a_{s,j}\in k$ for $l>n$.

\item Assume $i>n$.  Then the $(i,j)$ entry of $\Inn_A(Z_{r,s})$ is
given by\\
$a_{r,i-n}a_{s,j}-a_{n+s+r-1,i-n}a_{1,j}-a_{n+s+r-2,i-n}a_{2,j}-...-
a_{n+2,i-n}a_{s+r-2,j}-\\ a_{n+1,i-n}a_{s+r-1,j}+a_{1,i-n}a_{n+s+r-1,j}-
a_{2,i-n}a_{n+s+r-2,j}+...+a_{s+r-2,i-n}a_{n+2,j}+a_{s+r-1,i-n}a_{n+1,j}
+a_{n+r+s,i-n}a_{r+s,j}-a_{n+r+s+1,i-n}a_{r+s+1,j}-...-a_{2n,i-n}a_{n,j}
+a_{r+s,i-n}a_{n+r+s,j}+a_{r+s+1,i-n}a_{n+r+s+1,j}+...+a_{n,i-n}a_{2n,j}$\\
Note that the $(i,j)$ entry of $\Inn_A(Z_{r,s})$ for $i>n$ is the
negative of the $(i,j)$ entry of $\Inn_A(Z_{r,s})$ for $i\leq n$ with
the simple change that $n+i$ becomes $i-n$.  Again we have that
the $(i,j)$ entry of $\Inn_A(\bar{Z}_{r,s})$ is the negative of the
$(i,j)$ entry of $\Inn_A({Z_{r,s}})$ with the exception of
$a_{r,i-n}a_{s,j}$ which remains positive.  Hence as in the
previous case the $(i,j)$ entry of $\Inn_A{Z_{r,s}}+
\Inn_A(\bar{Z}_{r,s})$, gives us $2a_{r,i-n}a_{s,j}\in k$.  Since we
assumed that $i>n$ we can conclude that $a_{r,l}a_{s,j}\in k$ for
$l<n$. \\
 Combining a and b we have that $a_{r,i}a_{s,j}\in k$
for $r+s<n+1.$\\
\end{enumerate}

\item \textbf{Subcase 2}: Suppose $r+s>n+1$. \\  Let 
$$U_{r,s}=\begin{pmatrix} I_{ -n+(r+s-1)} & 0\\0 &
(T_ {2n-(r+s-1)}\end{pmatrix}$$ and
$$V_{r,s}=\begin{pmatrix} U_{r,s} & 0\\E_{r,s} &
U_{r,s}\end{pmatrix}.$$ 
 $V_{r,s}\in G$ so $\Inn_A(V_{r,s})\in G$ since $\Inn_A$ keeps $G$ invariant.

\begin{enumerate}
\item  Suppose $i\leq n$.  Then the $(i,j)$ entry of $\Inn_A(V_{r,s})$
is given by\\
$-a_{r,n+i}a_{s,j}+a_{2n,n+i}a_{s+r-n,j}+a_{2n-1,n+i}a_{s+r-n+1,j}+...+
a_{s+r,n+i}a_{n,j}-\\ a_{2n,j}a_{s+r-n,n+i}-a_{2n-1,j}a_{s+r-n+1,n+i}-...-
a_{s+r,j}a_{n,n+i}+a_{n+1,n+i}a_{1,j}+\\ a_{n+2,n+i}a_{2,j}+...+
a_{r+s-1,n+i}a_{(r+s-1)-n,j}-a_{1,n+i}a_{n+1,j}-a_{2,n+i}a_{n+2,j}-...-
a_{(r+s-1)-n,n+i}a_{r+s-1,j}.$\\
Let $\bar{V}_{r,s}=\begin{pmatrix} -U_{r,s} & 0\\E_{r,s} &
-U_{r,s}\end{pmatrix}.$  Now $\bar{V_{r,s}}\in G$ which implies
that $\Inn_A(\bar{V}_{r,s})\in G$.  The $(i,j)$ entry of
$\Inn_A(\bar{V_{r,s}})$ is the negative of the (i,j) entry of
$\Inn_A(V_{r,s})$ with the exception of the term $-a_{r,n+i}a_{s,j}$
which remains negative.  Hence the $(i,j)$ entry of
$\Inn_A(V_{r,s})+\Inn_A(\bar{V}_{r,s}), -2a_{r,n+i}a_{s,j}$ is in $k$.  Since
we assumed $i\leq n$ we have $a_{r,l}a_{s,j}\in k$ for $l>n$.

\item Assume $i>n$. As in the previous case, the  $(i,j)$ entry of
$\Inn_A(V_{r,s})$ for $i>n$ is the negative of the $(i,j)$ entry of
$\Inn_A(V_{r,s})$ for $i\leq n$ with the simple change that $n+i$
becomes $i-n$.  Again we have that the $(i,j)$ entry of
$\Inn_A(\bar{V}_{r,s})$ is the negative of the $(i,j)$ entry of
$\Inn_A({V_{r,s}})$ with the exception of $a_{r,i-n}a_{s,j}$ which
remains positive.  Hence as in the previous case the $(i,j)$ entry
of $\Inn_A{Z_{r,s}}+ \Inn_A(\bar{Z_{r,s}})$, gives us
$2a_{r,i-n}a_{s,j}$.  Since we assumed that $i>n$ we can conclude
that $a_{r,l}a_{s,j}\in k$ for $l<n$. \\
 Combining a and b we
have that $a_{r,i}a_{s,j}\in k$ for $r+s<n+1$
\end{enumerate}

\item  \textbf{Subcase 3:} Suppose $r+s=n+1$.  Here we choose $W_{r,s}=\begin{pmatrix} T_n & 0\\E_{r,s} & T_n\end{pmatrix}$.  Now
$W_{r,s}\in G$ and hence, $\Inn_A(W_{r,s}) \in G$ since $\Inn_A$ keeps
$G$ invariant.

\begin{enumerate}

\item Suppose $i\leq n$  Then the $(i,j)$ entry of
$\Inn_A(W_{r,s})$ is given by\\
$-a_{r,n+i}a_{s,j}+a_{2n,n+i}a_{1,j}+a_{2n-1,n+i}a_{2,j}+...+
a_{n+1,n+i}a_{n,j}-a_{2n,j}a_{1,n+i}+a_{2n-1,j}a_{2,n+i}+...+
a_{n+1,j}a_{n,n+i}$.\\
Let $\bar{W}_{r,s}= \begin{pmatrix} -T_n & 0\\E_{r,s} &
-T_n\end{pmatrix}$.  $\bar{W}_{r,s}\in G$ which means that
$\Inn_A(\bar{W}_{r,s})\in G$.  The $(i,j)$ entry of
$\Inn_A(\bar{W}_{r,s})$ is the negative of the $(i,j)$ entry of
$\Inn_A(W_{r,s})$ with the exception that the term
$-a_{r,n+i}a_{s,j}$ which remains negative.  Using the fact that
$\Inn_A(W_{r,s})+\Inn_A(\bar{W}_{r,s})\in G$ we have that the term
$-2a_{r,n+i}a_{s,j}\in k$.  However, since we assumed that $i\leq
n$ we have that $a_{r,l}a_{s,j}\in k$ for $l>n$.

\item  The case where $i>n$ follows exactly as above by simply
changing the signs of each term and replacing $n+i$ by $i-n$.\\
\end{enumerate}
\end{enumerate}

Combining Subcases 1,2, and 3 gives us $a_{r,i}a_{s,j}\in k$ for
$r,s>n.$\\

\noindent \textbf{CASE II: }Suppose $r,s>n$.  Without loss of generality assume $r<s$.\\
\begin{enumerate}
\item  \textbf{Subcase 1:} Suppose $r+s-2n<n+1$.\\
  Let $Z'_{r,s}=\begin{pmatrix}
Y_{r-n,s-n} & E_{r-n,s-n}\\0 & Y_{r-n,s-n}\end{pmatrix}$\\

\begin{enumerate}
\item Suppose $i\leq n$, Since $Z'_{r,s}\in G$,
$\Inn_A(Z'_{r,s})$ must lie in $G$ and hence its $(i,j)$ entry of\\
$a_{r,n+i}a_{s,j}+a_{n+s+r-1,n+i}a_{1,j}+a_{n+s+r-2,n+i}a_{2,j}+...+
a_{n+2,n+i}a_{s+r-2,j}+\\ a_{n+1,n+i}a_{s+r-1,j}-a_{1,n+i}a_{n+s+r-1,j}-
a_{2,n+i}a_{n+s+r-2,j}-...-a_{s+r-2,n+i}a_{n+2,j}-a_{s+r-1,n+i}a_{n+1,j}
+a_{n+r+s,n+i}a_{r+s,j}+a_{n+r+s+1,n+i}a_{r+s+1,j}+...+a_{2n,n+i}a_{n,j}
-a_{r+s,n+i}a_{n+r+s,j}-a_{r+s+1,n+i}a_{n+r+s+1,j}-...-a_{n,n+i}a_{2n,j}$\\
is in $k$.  Note that the $(i,j)$ entry of $\Inn_A(Z'_{r,s})$ is
precisely the $(i,j)$ entry of $\Inn_A(Z_{r,s})$ given in part I with
the exception of the first term. Let
$\bar{Z'}_{r,s}=\begin{pmatrix} -Y_{r-n,s-n} & E_{r-n,s-n}\\0 &
-Y_{r-n,s-n}\end{pmatrix}$.   The $(i,j)$ entry of $\Inn_A(\bar{Z'}_{r,s})$
is the negative of the $(i,j)$ entry of $\Inn_A(Z'_{r,s})$ excluding
the term $a_{r,n+i}a_{s,j}$ which remains positive. Hence the
$(i,j)$ entry of $\Inn_A(Z'_{r,s})+\Inn_A(\bar{Z'}_{r,s}), given by
2a_{r,n+i}a_{s,j}$ lies in $k$. Since we assumed $i\leq n$ we have
$a_{r,l}a_{s,j}\in k$ for $l>n$.

\item As in the previous cases, for $i>n$ the proof follows
exactly as above by simply changing the signs of each term and
replacing $n+i$ by $i-n$.  You will get that the $(i,j)$ entry of
$\Inn_A(Z'_{r,s})+\Inn_A(\bar{Z'}_{r,s})$ yields $-2a_{r,i-n}a_{s,j}\in
k$.  Or more specifically, $a_{r,l}a_{s,j}\in k$ for $l\leq n$.\\

Combining a and b gives $a_{r,i}a_{s,j}\in k$ for $r+s-2n<n+1$\\

\end{enumerate}

\item  \textbf{Subcase 2:} Suppose $r+s-2n>n+1$

\begin{enumerate}
\item Let $V'_{r,s}=\begin{pmatrix} U_{r-n,s-n} & E_{r-n,s-n}\\0 &
U_{r-n,s-n}\end{pmatrix}$. Now
$\Inn_A(V'_{r,s})$ must lie in G and hence its $(i,j)$ entry of\\
$a_{r,n+i}a_{s,j}+a_{2n,n+i}a_{s+r-n,j}+a_{2n-1,n+i}a_{s+r-n+1,j}+...+
a_{s+r,n+i}a_{n,j}-\\ a_{2n,j}a_{s+r-n,n+i}-a_{2n-1,j}a_{s+r-n+1,n+i}-...-
a_{s+r,j}a_{n,n+i}+a_{n+1,n+i}a_{1,j}+\\ a_{n+2,n+i}a_{2,j}+...+
a_{r+s-1,n+i}a_{(r+s-1)-n,j}-a_{1,n+i}a_{n+1,j}-a_{2,n+i}a_{n+2,j}-...-
a_{(r+s-1)-n,n+i}a_{r+s-1,j}$\\  must lie in $k$.  If we define
$\bar{V'}_{r,s}=\begin{pmatrix} -U_{r,s} & E_{r-n,s-n}\\0 &
-U_{r,s}\end{pmatrix}$, which is in $G$, then we see that the
$(i,j)$ entry of $\Inn_A(\bar{V'}_{r,s})$ is the negative of the
$(i,j)$ entry of $\Inn_A(V'_{r,s})$ excluding the term
$a_{r,n+i}a_{s,j}$ which remains positive. Hence the $(i,j)$ entry
of $\Inn_A(V'_{r,s})+\Inn_A(\bar{V'}_{r,s}), 2a_{r,n+i}a_{s,j}$ is in$ k$.
Since we assumed $i\leq n$ we have $a_{r,l}a_{s,j}\in k$ for
$l>n$.

\item  Again as in the previous cases, for $i>n$  the proof
follows exactly as above by simply changing the signs of each term
and replacing $n+i$ by $i-n$.  You will get that the $(i,j)$ entry
of $\Inn_A(V'_{r,s})+\Inn_A(\bar{V'}_{r,s})$ yields that the term
$-2a_{r,i-n}a_{s,j}$ is in $k$.  Or more specifically,
$a_{r,l}a_{s,j}\in k$ for $l\leq n.$\\

Combining a and b gives $a_{r,i}a_{s,j}\in k$ for $r+s-2n>n+1$\\

\end{enumerate}

\item \textbf{Subcase 3}: Suppose $r+s-2n=n+1.$  Let $W'_{r,s}=\begin{pmatrix} T_n &
E_{r-n,s-n}\\0 & T_n\end{pmatrix}.$  Now $W'_{r,s}\in G$ and thus
$\Inn_A(W'_{r,s})\in G$.

\begin{enumerate}
\item Suppose $i\leq n$, then the $(i,j)$ entry of $\Inn_A(W'_{r,s})$
is
given by \\
$a_{r,n+i}a_{s,j}+a_{2n,n+i}a_{1,j}+a_{2n-1,n+i}a_{2,j}+...+
a_{n+1,n+i}a_{n,j}-a_{2n,j}a_{1,n+i}+\\ a_{2n-1,j}a_{2,n+i}+...+
a_{n+1,j}a_{n,n+i}$.\\
If we let $\bar{W'}_{r,s}=\begin{pmatrix} -T_n & E_{r-n,s-n}\\0 &
-T_n\end{pmatrix}$, then we see that the $(i,j)$ entry is simply the
negative of the $(i,j)$ entry of $\Inn_A(W'_{r,s})$ excluding the term
$a_{r,n+i}a_{s,j}$ which remains positive. Hence the $(i,j)$ entry
of $\Inn_A(W'_{r,s})+\Inn_A(\bar{W'}_{r,s}), 2a_{r,n+i}a_{s,j}$ is in $k$.
Since we assumed $i\leq n$ we have $a_{r,l}a_{s,j}\in k$ for
$l>n$.

\item  Again as in the previous cases, for $i>n$  the proof
follows exactly as above by simply changing the signs of each term
and replacing $n+i$ by $i-n$.  You will get that the $(i,j)$ entry
of $\Inn_A(W'_{r,s})+\Inn_A(\bar{W'_{r,s}})$ yields that the term
$-2a_{r,i-n}a_{s,j}$ is in $k$.  Or more specifically,
$a_{r,l}a_{s,j}\in k$ for $l\leq n$.\\

Combining a and b gives $a_{r,i}a_{s,j}\in k$ for $r+s-2n>n+1$\\

\end{enumerate}

\end{enumerate}

\noindent \textbf{CASE III: }Suppose $r\leq n$ and $s>n$\\
\begin{enumerate}
\item \textbf{Subcase 1:} Suppose $r+s<2n+1$. \\
 Let
$M_{r,s}=\begin{pmatrix} E_{s-n,r} & Y_{s-n,r}\\-Y_{s-n,r} &
0\end{pmatrix}.$
Now $M_{r,s}\in G$ and thus $\Inn_A(M_{r,s})\in G$ by assumption.
\begin{enumerate}
\item Suppose $i\leq n$.  Now the $(i,j)$ entry of $\Inn_A(M_{r,s})$ is
given by \\
$a_{r,j}a_{s,n+i}+a_{s+r,n+i}a_{s+r,j}+a_{s+r+1,n+i}a_{s+r+1,j}+...+
a_{2n,n+i}a_{2n,j}+\\ a_{s+r-n,n+i}a_{s+r-n,j}+a_{s+r-n+1,n+i}a_{s+r-n+1,j}
+...+a_{n,n+i}a_{n,j}+a_{s+r-1,n+i}a_{n+1,j}+a_{s+r-2,n+i}a_{n+2,j}
+...+a_{n+1,n+i}a_{s+r-1,j}+a_{s-n+r-1,n+i}a_{1,j}+a_{s-n+r-2,n+i}a_{2,j}
+...+a_{1,n+i}a_{s-n+r-1,j}.$\\
Now define $\bar{M}_{r,s}=\begin{pmatrix} E_{s-n,r} &
-Y_{s-n,r}\\Y_{s-n,r} & 0\end{pmatrix}$, then
$\bar{M_{r,s}}\in G$ and therefore $\Inn_A(\bar{M}_{r,s})\in G$.  In
addition the $(i,j)$ entry of $\bar{M}_{r,s}$ is the negative of the
$(i,j)$ entry of $\Inn_A(M_{r,s})$ with the exception of the term
$a_{r,j}a_{s,i+n}$ which remains positive.  The sum $\Inn_A(M_{r,s})+ \Inn_A(\bar{M}_{r,s})\in G$
and thus its $(i,j)$ entry of $2a_{r,j}a_{s,i+n}\in k$.  Since we
assumed that $i\leq n$ this gives us $a_{r,j}a_{s,l}\in k$ for
$l>n$.

\item As in the previous cases, for $i>n$  the proof follows
exactly as above by simply changing the signs of each term and
replacing $n+i$ by $i-n$.  You will get that the $(i,j)$ entry of
$\Inn_A(M_{r,s})+\Inn_A(\bar{M}_{r,s})$ yields that the term
$-2a_{r,j}a_{s,i-n}$ is in $k$. Or more specifically,
$a_{r,j}a_{s,l}\in k$ for $l\leq n.$\\

Combining a and b gives $a_{r,j}a_{s,l}\in k$ for $r+s<2n+1.$\\

\end{enumerate}

\item \textbf{Subcase 2: } Suppose $r+s>2n+1$.   \\
 With $N_{r,s}=\begin{pmatrix} E_{s-n,r} &
U_{s-n,r}\\-U_{s-n,r} & 0\end{pmatrix}$ it is seen that $N_{r,s}\in G$ and thus by assumption
$\Inn_A(N_{r,s})\in G$

\begin{enumerate}
\item Suppose $i\leq n$ then the $(i,j)$ entry of $\Inn_A(N_{r,s})$ is
given by\\
$a_{r,j}a_{s,n+i}+a_{2n,n+i}a_{s+r-n,j}+a_{2n-1,n+i}a_{s+r-n+1,j}
+...+a_{s+r-n,n+i}a_{2n,j}+\\a_{n,n+i}a_{s+r-2n,j}+a_{n-1,n+i}a_{s+r-2n+1,j}
+...+a_{s+r-2n,n+i}a_{n,j}+a_{n+1,n+i}a_{n+1,j}+a_{n+2,n+i}a_{n+2,j}
+...+a_{-n+r+s-1,n+i}a_{-n+r+s-1,j}+a_{1,n+i}a_{1,j}+a_{2,n+i}a_{2,j}
+...+a_{-2n+r+s-1,n+i}a_{-2n+r+s-1,j}$\\
Define $\bar{N}_{r,s}=\begin{pmatrix} E_{s-n,r} &
-U_{s-n,r}\\U_{s-n,r} & 0\end{pmatrix}.$
We again can make use of the fact that $\bar{N}_{r,s}\in G$
implies that $\Inn_A(\bar{N}_{r,s})\in G$.  Now the $(i,j)$ entry of
$\Inn_A(\bar{N}_{r,s})$ is the negative of the $(i,j)$ entry of
$\Inn_A(N_{r,s})$ with the exception of the the term
$a_{r,j}a_{s,n+i}$ which remains positive.  Hence the $(i,j)$ entry of
$\Inn_A(\bar{N}_{r,s})+\Inn_A(N_{r,s})$  given by $2a_{r,j}a_{s,n+i}$
must lie in $k$.  Furthermore, since we assumed that $i\leq n$
we can conclude that $a_{r,j}a_{s,l}\in k$ for $l>n$.

\item As in the previous cases, if $i>n$ the proof follows exactly
as above by simply changing the signs of each term and replacing
$n+i$ by $i-n$.  You will get that the $(i,j)$ entry of
$\Inn_A(N_{r,s})+\Inn_A(\bar{N_{r,s}})$ yields that the term
$-2a_{r,j}a_{s,i-n}$ is in $k$. Or more specifically,
$a_{r,j}a_{s,l}\in k$ for $l\leq n.$\\

Combining a and b gives $a_{r,j}a_{s,l}\in k$ for $r+s<2n+1$\\

\end{enumerate}

\item \textbf{Subcase 3:}  Suppose $r+s=2n+1$.  Let $F_{r,s}=\begin{pmatrix}
E_{s-n,r} &T_n\\T_n & 0\end{pmatrix}$.\\
 Now
$F_{r,s}\in G$ and therefore $\Inn_A(F_{r,s})\in G$ since
$\Inn_A$ keeps G invariant.
\begin{enumerate}
\item Suppose $i\leq n$.  Then the $(i,j)$ entry of
$\Inn_A(F_{r,s})$ is given by\\
$a_{r,j}a_{s,n+i}+a_{1,n+i}a_{1,j}+a_{2,n+i}a_{2,j}+a_{2,n+i}a_{3,j}+...+
a_{2n,n+i}a_{2n,j}$\\
Let $\bar{F}_{r,s}=\begin{pmatrix} E_{s-n,r} & -T_n\\-T_n & 0\end{pmatrix}.$
Then since $\bar{F}_{r,s}\in G$ we have that
$\Inn_A(\bar{F}_{r,s})\in G$.  More importantly, the $(i,j)$ entry of
$\Inn_A(\bar{F}_{r,s})$ is the negative of the $(i,j)$ entry of
$\Inn_A(F_{r,s})$ with the exception that the term
$a_{r,j}a_{s,i+n}$ remains positive.  Again using the fact that
$\Inn_A(\bar{F}_{r,s})+\Inn_A(F_{r,s})\in G$ we have that its
$(i,j)$ entry of $2a_{r,j}a_{s,n+i}\in k$.  Since we assumed $i\leq
n$ we have that $a_{r,j}a_{s,l}\in k$ for $l>n$.\\

\item As in the previous cases, if $i>n$ the proof follows exactly
as above by simply changing the signs of each term and replacing
$n+i$ by $i-n$.  You will get that the $(i,j)$ entry of
$\Inn_A(\bar{F}_{r,s})+\Inn_A(F_{r,s})\in G$ yields that the term
$-2a_{r,j}a_{s,i-n}$ is in $k$. Or more specifically,
$a_{r,j}a_{s,l}\in k$ for $l\leq n$\\

Combining subcases a and b gives us $a_{r,j}a_{s,l}\in k$ for
$r+s=2n+1.$\\

Combining cases 1,2, and 3 gives us $a_{r,i}a_{s,j}\in k$ for
$r\leq n$ and $s>n$.

\end{enumerate}

\end{enumerate}

Cases I, II and III show that $a_{r,i}a_{s,j}\in k$. From this it is clear that $k[a_{is}] = k[a_{jt}]$ for all $i,j,s,t$ (assuming that $a_{is}$ and $a_{jt}$ are both nonzero). So, let $\alpha = a_{is}^2$ where $a_{is}$ is a fixed nonzero entry of $A$. Then, we have shown that all the entries of $A$ are in $k[\sqrt{\alpha}].$ This means that $A \in \Sp(2n,k[\sqrt{\alpha}])$, and all of the entries of $A$ are $k$-multiples of $\sqrt{\alpha}$, as desired. 
\end{proof}

%%%%%%%%%%%%%%%%%%%%%%%%%%%%%%%%%%%%%

\section{Involutions of $\Sp(2n,k)$}

We now begin to focus on involutions and the classification of their isomorphy classes. We will distinguish different types of involutions. First, we note that for some involutions, $\phi$, there exists $A \in \Sp(2n,k)$ such that $\phi = \Inn_A$, but not in all cases. Sometimes we must settle for $A \in \Sp(2n,k[\sqrt{\alpha}]) \setminus \Sp(2n,k)$.

This is not the only way in which we can distinguish between different types of involutions. If $\Inn_A$ is an involution, then $\Inn_{A^2} = (\Inn_A)^2$ is the identity map. We know from above that this means that $A^2 = \gamma I$ for some $\gamma \in \overline{k}.$ But, we know that $A$ is symplectic. So, $A^2$ is also symplectic. That means that $(A^2)^TJ(A^2) = J$, which implies $(\gamma I)^TJ (\gamma I) = J$, which means $\gamma^2 = 1$. So, $\gamma = \pm 1.$ Thus, we can also distinguish between different types of involutions by seeing if $A^2 = I$ or $A^2 = -I$. This gives the four types of involutions, which are outlined in Table \ref{InvDef}.

\begin{table}[h] 
\centering
\caption {The various possible types of involutions of $\Sp(2n,k)$}  \label{InvDef}
\begin{tabular}[t]{|c||c|c|}
\hline  & $ A \in \Sp(2n,k)$ & $A \in \Sp(2n,k[\sqrt{\alpha}]) \setminus \Sp(n,k)$ \\
\hline  \hline $A^2 = I$ & Type 1 & Type 2 \\ 
\hline  $A^2 = -I$ & Type 3 & Type 4 \\
\hline
\end{tabular}
\end{table}

\subsection{Type 1 Involutions}

We first characterize the matrices that induce Type 1 involutions in the following lemma.

\begin{lem}
\label{Type1Class}
Suppose $\theta$ is a Type 1 involution of $\Sp(2n,k)$. Then, $$A = X \left(\begin{array}{cccc}I_{\frac{s}{2}} & 0 & 0 & 0 \\0 & -I_{\frac{t}{2}} & 0 & 0 \\0 & 0 & I_{\frac{s}{2}} & 0 \\0 & 0 & 0 & -I_{\frac{t}{2}}\end{array}\right) X^{-1}$$ where $s+t =2n$ and $X^TJX = J$. That is, $X \in \Sp(2n,k)$.
\end{lem}

\begin{proof}

 Since $\Inn_A^2 = I$ and $A \in \Sp(2n,k)$, then it follows that $A^2 = I$. So, all eigenvalues of $A$ are $\pm 1$. Since there are no repeated roots in the minimal polynomial of $A$, then we see that $A$ is diagonalizable. Let $s = \dim(E(A,1))$ and $t = \dim(E(A,-1))$, and observe that $s+t = 2n$ since $A$ is diagonalizable. We will first show that both $s$ and $t$ must be even. To do this, we proceed by contradiction and assume that $s$ and $t$ are both odd. So, there exists some $Y \in \Gl(n,\overline{k})$ such that $Y^{-1}AY = \left(\begin{smallmatrix}I_s & 0 \\0 & -I_t\end{smallmatrix}\right).$ Since $A$ is symplectic, then it follows that 
\begin{align*}
J &= A^TJA\\ 
&= \left(Y\left(\begin{array}{cc}I_s & 0 \\0 & -I_t\end{array}\right)Y^{-1}\right)^T JY\left(\begin{array}{cc}I_s & 0 \\0 & -I_t\end{array}\right)Y^{-1}\\
&= (Y^{-1})^T\left(\begin{array}{cc}I_s & 0 \\0 & -I_t\end{array}\right)Y^TJY \left(\begin{array}{cc}I_s & 0 \\0 & -I_t\end{array}\right)Y^{-1}.
\end{align*}
 
 This implies that $$ \left(\begin{array}{cc}I_s & 0 \\0 & -I_t\end{array}\right)Y^TJY = (Y^TJY) \left(\begin{array}{cc}I_s & 0 \\0 & -I_t\end{array}\right),$$ where $Y^TJY$ is an invertible skew-symmetric matrix. So, $Y^TJY = \left(\begin{smallmatrix}Y_1 & 0 \\0 & Y_2\end{smallmatrix}\right)$ for some invertible skew symmetric matrices $Y_1$ and $Y_2$, which are $s \times s$ and $t \times t$, respectively. But odd dimensional skew symmetric matrices cannot be invertible, so this is a contradiction. Thus, $s$ and $t$ must be even.

We now wish to construct bases for $E(A,1)$ and $E(A,-1)$ such that all the vectors lie in $k^n$.  Let $\{z_1,...,z_n\}$ be a basis for $k^n$. For each $i$, let $u_i = (A+I)z_i.$ Note that $$Au_i = A(A+I)z_i = (A+I)z_i = u_i.$$ So, $\{u_1,...,u_n\}$ must span $E(A,1)$. Thus, we can appropriately choose $s$ of these vectors and form a basis for $E(A,1)$. Label these basis vectors as $y_1,...,y_{\frac{s}{2}}, y_{n+1},...,y_{n+\frac{s}{2}}$. We can similarly form a basis for $E(A,-1)$. We shall call these vectors $y_{\frac{s}{2}+1},...,y_{n},y_{n+\frac{s}{2}+1},...,y_{2n}$. Let $Y$ be the matrix with the vectors $y_1,...,y_{2n}$ as its columns. Then, by construction, $$Y^{-1}AY =  \left(\begin{array}{cccc}I_{\frac{s}{2}} & 0 & 0 & 0 \\0 & -I_{\frac{t}{2}} & 0 & 0 \\0 & 0 & I_{\frac{s}{2}} & 0 \\0 & 0 & 0 & -I_{\frac{t}{2}}\end{array}\right).$$ We can rearrange to get $$A = Y \left(\begin{array}{cccc}I_{\frac{s}{2}} & 0 & 0 & 0 \\0 & -I_{\frac{t}{2}} & 0 & 0 \\0 & 0 & I_{\frac{s}{2}} & 0 \\0 & 0 & 0 & -I_{\frac{t}{2}}\end{array}\right)Y^{-1}.$$

Recall that $A^T = JAJ^{-1}$, since $A \in \Sp(2n,k)$. So, 

$$\left( Y \left(\begin{array}{cccc}I_{\frac{s}{2}} & 0 & 0 & 0 \\0 & -I_{\frac{t}{2}} & 0 & 0 \\0 & 0 & I_{\frac{s}{2}} & 0 \\0 & 0 & 0 & -I_{\frac{t}{2}}\end{array}\right)Y^{-1} \right)^T = J \left( Y \left(\begin{array}{cccc}I_{\frac{s}{2}} & 0 & 0 & 0 \\0 & -I_{\frac{t}{2}} & 0 & 0 \\0 & 0 & I_{\frac{s}{2}} & 0 \\0 & 0 & 0 & -I_{\frac{t}{2}}\end{array}\right) Y^{-1} \right) J^{-1}.$$

This implies 

$$(Y^{-1})^T \left(\begin{array}{cccc}I_{\frac{s}{2}} & 0 & 0 & 0 \\0 & -I_{\frac{t}{2}} & 0 & 0 \\0 & 0 & I_{\frac{s}{2}} & 0 \\0 & 0 & 0 & -I_{\frac{t}{2}}\end{array}\right)Y^T = JY \left(\begin{array}{cccc}I_{\frac{s}{2}} & 0 & 0 & 0 \\0 & -I_{\frac{t}{2}} & 0 & 0 \\0 & 0 & I_{\frac{s}{2}} & 0 \\0 & 0 & 0 & -I_{\frac{t}{2}}\end{array}\right) (JY)^{-1},$$

which means

$$ \left(\begin{array}{cccc}I_{\frac{s}{2}} & 0 & 0 & 0 \\0 & -I_{\frac{t}{2}} & 0 & 0 \\0 & 0 & I_{\frac{s}{2}} & 0 \\0 & 0 & 0 & -I_{\frac{t}{2}}\end{array}\right) Y^TJY = Y^TJY\left(\begin{array}{cccc}I_{\frac{s}{2}} & 0 & 0 & 0 \\0 & -I_{\frac{t}{2}} & 0 & 0 \\0 & 0 & I_{\frac{s}{2}} & 0 \\0 & 0 & 0 & -I_{\frac{t}{2}}\end{array}\right) .$$

So, $Y^TJY = \left(\begin{smallmatrix}Y_1 & 0 & Y_2 & 0 \\0 & Y_3 & 0 & Y_4 \\-Y_2^T & 0 & Y_5 & 0 \\0 & -Y_4^T & 0 & Y_6\end{smallmatrix}\right)$, where $Y_1$ and $Y_5$ are $\frac{s}{2} \times \frac{s}{2}$ skew-symmetric matrices, $Y_3$ and $Y_6$ are $\frac{t}{2} \times \frac{t}{2}$ skew-symmetric matrices, $Y_2$ is a $\frac{s}{2} \times \frac{s}{2}$ matrix and $Y_4$ is a $\frac{t}{2} \times \frac{t}{2}$ matrix.

We can choose a permutation matrix $Q \in \oo(2n,k)$ such that $$A = YQ\left(\begin{array}{cc}I_s & 0 \\0 & -I_t\end{array}\right) Q^{-1}Y^{-1}$$ and $$Y^TJY = Q \left(\begin{array}{cccc}Y_1 & Y_2 & 0 & 0 \\-Y_2^T & Y_5 & 0 & 0 \\0 & 0 & Y_3 & Y_4 \\0 & 0 & -Y_4^T & Y_6\end{array}\right) Q^{-1}.$$ Let $Y_7 =  \left(\begin{smallmatrix}Y_1 & Y_2 \\-Y_2^T & Y_5\end{smallmatrix}\right)$ and $Y_8 = \left(\begin{smallmatrix}Y_3 & Y_4 \\-Y_4^T & Y_6\end{smallmatrix}\right).$ Note that both $Y_7$ and $Y_8$ are skew-symmetric. We can rearrange the above statement to be $$Q^TY^TJYQ = \left(\begin{array}{cc}Y_7 & 0 \\0 & Y_8\end{array}\right).$$

 It follows that there exists $N = \left(\begin{smallmatrix}N_1 & 0 \\0 & N_2\end{smallmatrix}\right) \in \Gl(n,k)$ such that $$N^TQ^TY^TJYQN = \left(\begin{array}{cccc}0 & I_{\frac{s}{2}} & 0 & 0 \\-I_{\frac{s}{2}} & 0 & 0 & 0 \\0 & 0 & 0 & I_{\frac{t}{2}} \\0 & 0 & -I_{\frac{t}{2}} & 0\end{array}\right).$$ 
 
 We see that we can again use the permutation matrix $Q$ to get  $$QN^TQ^TY^TJYQNQ^T = \left(\begin{array}{cc}0 & I_n \\-I_n & 0\end{array}\right) = J.$$ 
 
 Let $X = YQNQ^T$. Then, 
 \begin{align*}
 X \left(\begin{array}{cccc}I_{\frac{s}{2}} & 0 & 0 & 0 \\0 & -I_{\frac{t}{2}} & 0 & 0 \\0 & 0 & I_{\frac{s}{2}} & 0 \\0 & 0 & 0 & -I_{\frac{t}{2}}\end{array}\right) X^{-1} &= YQNQ^T\left(\begin{array}{cccc}I_{\frac{s}{2}} & 0 & 0 & 0 \\0 & -I_{\frac{t}{2}} & 0 & 0 \\0 & 0 & I_{\frac{s}{2}} & 0 \\0 & 0 & 0 & -I_{\frac{t}{2}}\end{array}\right) (YQNQ^T)^{-1}\\
 &= YQ \left(\begin{array}{cc}N_1 & 0 \\0 & N_2\end{array}\right) \left(\begin{array}{cc}I_s & 0 \\0 &- I_t\end{array}\right) \left(\begin{array}{cc}N_1^{-1} & 0 \\0 & N_2^{-1}\end{array}\right)Q^{-1}Y^{-1}\\
 &= YQ  \left(\begin{array}{cc}I_s & 0 \\0 &- I_t\end{array}\right) Q^{-1}Y^{-1}\\
 &= Y  \left(\begin{array}{cccc}I_{\frac{s}{2}} & 0 & 0 & 0 \\0 & -I_{\frac{t}{2}} & 0 & 0 \\0 & 0 & I_{\frac{s}{2}} & 0 \\0 & 0 & 0 & -I_{\frac{t}{2}}\end{array}\right) Y^{-1}\\
 &= A,
 \end{align*} 
 where $X^TJX = J.$ From this last observation, it follows that $X \in \Sp(2n,k)$.
\end{proof}

Using this characterization, we now find conditions on these involutions that are equivalent to isomorphy.

\begin{theorem}
Suppose $\Inn_A$ and $\Inn_B$ both induce Type 1 involutions for $\Sp(2n,k)$ for $A$ and $B \in \Sp(2n,k)$. Then, $\Inn_A$ and $\Inn_B$ are isomorphic over $\Sp(2n,k)$ if and only if the dimension of $E(A,1)$ equals the dimension of $E(B,1)$ or $E(B,-1)$.
\end{theorem} 

\begin{proof}

We first prove that  $\Inn_A$ is isomorphic to $\Inn_B$ over $\Sp(2n,k)$ is equivalent to $A$ being conjugate to $B$ or $-B$ over $\Sp(2n,k)$. Suppose $A$ is conjugate to $B$ over $\Sp(2n,k)$. Choose $Q \in \Sp(2n,k)$ such that $B = Q^{-1}AQ$. Then, for all $U \in \Sp(2n,k)$, we have 
\begin{align*}
Q^{-1}A^{-1}QUQ^{-1}AQ &= (Q^{-1}AQ)^{-1}U(Q^{-1}AQ)\\
&= B^{-1}UB.
\end{align*} 
So, $(\Inn_Q)^{-1}\Inn_A \Inn_Q = \Inn_B$. That is, $\Inn_A$ is isomorphic to $\Inn_B$ over $\Sp(2n,k)$. Likewise, if $A$ is conjugate to $-B$, then we can show $\Inn_A$ is isomorphic to $\Inn_B$ over $\Sp(2n,k)$. This argument is easily reversible.  

From this, it is clear that if  $\Inn_A$ and $\Inn_B$ are isomorphic over $\Sp(2n,k)$, then the dimension of $E(A,1)$ equals the dimension of $E(B,1)$ or $E(B,-1)$. We need only show the converse.

First, suppose that the dimension of $E(A,1)$ equals the dimension of $E(B,1)$. By the previous lemma, we can choose $X, Y \in \Sp(2n,k)$ such that 

$$X^{-1}AX = \left(\begin{array}{cccc}I_{\frac{s}{2}} & 0 & 0 & 0 \\0 & -I_{\frac{t}{2}} & 0 & 0 \\0 & 0 & I_{\frac{s}{2}} & 0 \\0 & 0 & 0 & -I_{\frac{t}{2}}\end{array}\right) = Y^{-1}BY.$$

Let $Q = XY^{-1}$. Note that $Q \in \Sp(2n,k)$. Then, we have $Q^{-1}AQ = B$, and we have already shown that this implies $\Inn_A$ is isomorphic to $\Inn_B$ over $\Sp(2n,k)$. 

If the dimension of $E(A,1)$ equals the dimension of $E(B,-1)$, then we can similarly show that there exists $Q \in \Sp(2n,k)$ such that $Q^{-1}AQ = - B$, which also implies $\Inn_A$ is isomorphic to $\Inn_B$ over $\Sp(2n,k)$. 

\end{proof}

From this theorem, the number of isomorphy classes of Type 1 involutions is clear. We note that this number is independent of the field $k$.

\begin{cor}
$\Sp(2n,k)$ has $\frac{n}{2}$ or $\frac{n-1}{2}$  isomorphy classes of Type 1 involutions. (Whichever is an integer.)
\end{cor}

\subsection{Type 2 Involutions}

We have a similar characterization of the matrices and isomorphy classes in the Type 2 case.  We first prove a result that characterizes the eigenvectors in the Type 2 case.

\begin{lem}

Suppose $A \in \Sp(2n,k[\sqrt{\alpha}],\beta) \setminus \Sp(2n,k,\beta) $ induces a Type 2 involution of $\Sp(n,k,\beta)$ where $\sqrt{\alpha} \not \in k$. Also suppose $x,y \in k^{2n}$ such that $x+\sqrt{\alpha}y \in E(A,-1)$. Then, $x-\sqrt{\alpha}y \in E(A,1)$. Likewise, if $u,v \in k^{2n}$ such that $u+\sqrt{\alpha}v \in E(A,1)$. Then, $u-\sqrt{\alpha}v \in E(A,-1)$. Further, $\dim(E(A,1))= \dim(E(A,-1))$.
\end{lem}

\begin{proof}

First, we observe that ``$\sqrt{\alpha}$-conjugation," similar to the familiar complex conjugation ($i$-conjugation), preserves multiplication. That is, $$(a+\sqrt{\alpha}b)(c+\sqrt{\alpha}d) = (ac+\alpha bd) + \sqrt{\alpha}(ad+bc)$$ and $$(a-\sqrt{\alpha}b)(c-\sqrt{\alpha}d) = (ac+\alpha bd) - \sqrt{\alpha}(ad+bc).$$ So, ``$\sqrt{\alpha}-$conjugation" will preserve multiplication on the matrix level as well. Because of this and since $$A(x+\sqrt{\alpha}y) = -x-\sqrt{\alpha}y,$$ then it follows that $$(-A)(x-\sqrt{\alpha}) = -x+\sqrt{\alpha}y.$$ We can multiply both sides to see that $$A(x-\sqrt{\alpha}) = x-\sqrt{\alpha}y.$$ That is, $x-\sqrt{\alpha}y \in E(A,1)$. This proves the first statement. An analogous argument proves the second.

To see that $\dim(E(A,1))= \dim(E(A,-1))$ is the case, note that the first statement tells us that $\dim(E(A,1)) \le \dim(E(A,-1))$, and that the second statement tells us that $\dim(E(A,1))\ge \dim(E(A,-1))$, since ``$\sqrt{\alpha}$-conjugation" is an invertible operator on $k[\sqrt{\alpha}]^n$.
\end{proof}

We are now able to characterize the Type 2 involutions. Note that this result combined with our results from the Type 1 case shows that if $n$ is odd, then $\Sp(2n,k)$ will not have any Type 2 involutions.

\begin{lem}
\label{Type2Class}
Suppose $\theta$ is a Type 2 involution of $\Sp(2n,k)$. Let $A$ be the symplectic matrix in $\Sp(2n,k[\sqrt{\alpha}])$ such that $\theta = \Inn_A$. Then, $$A = \frac{\sqrt{\alpha}}{\alpha} X  \left(\begin{array}{cc}0 & I_{n} \\ \alpha I_{n} & 0\end{array}\right)  X^{-1}$$ where $$X  = \left(\begin{array}{cccccccc}x_1 & x_2 & \cdots & x_{n} &y_1 & y_2 & \cdots & y_{n} \end{array}\right)\in \Gl(2n,k),$$ where 
 for each $i$, we have that $x_i+\sqrt{\alpha}y_i \in E(A,1)$ and $x_i-\sqrt{\alpha}y_i \in E(A,-1)$. Further, $$X^TJX = \frac{1}{2} \left(\begin{array}{cc}J & 0 \\0 & \frac{1}{\alpha} J\end{array}\right).$$
 
\end{lem}

\begin{proof}

We wish to construct bases for $E(A,1)$ and $E(A,-1)$ such that all the vectors lie in $k[\sqrt{\alpha}]^{2n}$. From the previous lemma, we know that $\dim(E(A,1)) = \dim(E(A,-1)) =  n.$ Since $\Inn_A$ is a Type 1 involution of $\Sp(2n,k[\alpha])$, then we can apply Lemma \ref{Type1Class} to find a basis $\{ x_1+\sqrt{\alpha}y_1,...,x_{n}+\sqrt{\alpha}y_{n} \}$ of $E(A,1)$, where $x_1,...,x_{n},y_1,...,y_{n} \in k^{2n}$. By the previous lemma, we know that$\{x_1-\sqrt{\alpha}y_1,...,x_{\frac{n}{2}}-\sqrt{\alpha}y_{\frac{n}{2}}\}$ must be a basis for $E(A,1)$. Further, based on  Lemma \ref{Type1Class}, we can assume that these vectors are chosen so that if 

$$Y = (x_1+\sqrt{\alpha}y_1,...,x_{\frac{n}{2}}+\sqrt{\alpha}y_{\frac{n}{2}}, x_1-\sqrt{\alpha}y_1,...,x_{\frac{n}{2}}-\sqrt{\alpha}y_{\frac{n}{2}}, x_{\frac{n}{2}+1}+\sqrt{\alpha}y_{\frac{n}{2}+1},...$$ $$...,x_{n}+\sqrt{\alpha}y_{n}, x_{\frac{n}{2}+1}-\sqrt{\alpha}y_{\frac{n}{2}+1},...,x_{n}-\sqrt{\alpha}y_{n}),$$

then we know that $$A = Y \left(\begin{array}{cccc}I_{\frac{n}{2}} & 0 & 0 & 0 \\0 & -I_{\frac{n}{2}} & 0 & 0 \\0 & 0 & I_{\frac{n}{2}} & 0 \\0 & 0 & 0 & -I_{\frac{n}{2}}\end{array}\right) Y^{-1}$$ where $Y^TJY = J$.

Let $X  = \left(\begin{smallmatrix}x_1 & x_2 & \cdots & x_{\frac{n}{2}} &y_1 & y_2 & \cdots & y_{\frac{n}{2}} \end{smallmatrix}\right)\in \Gl(n,k).$

We now make a couple of observations. Suppose $u = x+\sqrt{\alpha} y$ is a 1-eigenvector of $A$ such that $x,y \in k^n$. Then, we know $v = x-\sqrt{\alpha} y$ is a $-1$-eigenvector of $A$. Observe that $$Ax = \frac{1}{2}A(u+v) = \frac{1}{2}(u-v) = \sqrt{\alpha} y.$$ It follows from this that $$Ay = \frac{\sqrt{\alpha}}{\alpha}x.$$ 

Since $Ax =  \sqrt{\alpha} y$ and $Ay = \frac{\sqrt{\alpha}}{\alpha}x$, then it follows that $$X^{-1}AX = \left(\begin{array}{cc}0 & \frac{\sqrt{\alpha}}{\alpha}I_{\frac{n}{2}} \\ \sqrt{\alpha}I_{\frac{n}{2}} & 0\end{array}\right).$$ Rearranging this, we see that $$A = \frac{\sqrt{\alpha}}{\alpha} X  \left(\begin{array}{cc}0 & I_{\frac{n}{2}} \\ \alpha I_{\frac{n}{2}} & 0\end{array}\right)  X^{-1}.$$

Now, we need only prove the last statement to prove the Lemma. Since $Y^TJY = J$, then we know that if $1 \le i \le \frac{n}{2}$ and $j \ne \frac{n}{2}+i$, then 

$$0 = \beta(x_i+\sqrt{\alpha}y_i, x_j+\sqrt{\alpha}y_j) = (\beta(x_i, x_j)+\alpha\beta(y_i, y_j))+\sqrt{\alpha}(\beta(x_i,y_j)+\beta(x_j,y_i))$$

and that 

$$0 = \beta(x_i+\sqrt{\alpha}y_i, x_j-\sqrt{\alpha}y_j) = (\beta(x_i, x_j)-\alpha\beta(y_i, y_j))+\sqrt{\alpha}(-\beta(x_i,y_j)+\beta(x_j,y_i)).$$

So, we have that $\beta(x_i, x_j)+\alpha\beta(y_i, y_j) = 0$, $\beta(x_i,y_j)+\beta(x_j,y_i) = 0$, $\beta(x_i, x_j)-\alpha\beta(y_i, y_j) = 0$, and $-\beta(x_i,y_j)+\beta(x_j,y_i) = 0$. It follows from this that when $1 \le i \le \frac{n}{2}$ and $j \ne \frac{n}{2}+i$, we have 

$$\beta(x_i, x_j) = \beta(y_i, y_j) = \beta(x_i, y_j) = \beta(y_i, x_j) = 0.$$

Now suppose that $1 \le i \le \frac{n}{2}$ and $j = \frac{n}{2}+i$. Then, we have 

$$1 = \beta(x_i+\sqrt{\alpha}y_i, x_j+\sqrt{\alpha}y_j) = (\beta(x_i, x_j)+\alpha\beta(y_i, y_j))+\sqrt{\alpha}(\beta(x_i,y_j)+\beta(x_j,y_i))$$

and that 

$$0 = \beta(x_i+\sqrt{\alpha}y_i, x_j-\sqrt{\alpha}y_j) = (\beta(x_i, x_j)-\alpha\beta(y_i, y_j))+\sqrt{\alpha}(-\beta(x_i,y_j)+\beta(x_j,y_i)).$$

Similar to the first case, we have that $\beta(x_i, y_j) = 0 = \beta(y_i, x_j) = 0$, and we have 

$$1 = \beta(x_i, x_j)+\alpha\beta(y_i, y_j)$$

and 

$$0 = \beta(x_i, x_j)-\alpha\beta(y_i, y_j).$$

Thus, when $1 \le i \le \frac{n}{2}$ and $j = \frac{n}{2}+i$, we have that $\beta(x_i, x_j) = \frac{1}{2}$ and $\beta(y_i, y_j) = \frac{1}{2\alpha}$. So, we have that $X^TJX = \frac{1}{2} \left(\begin{smallmatrix}J & 0 \\0 & \frac{1}{\alpha} J\end{smallmatrix}\right).$

\end{proof}

We now consider a couple of examples of Type 2 involutions.

\begin{beisp}
Consider the matrix $$A = \frac{\sqrt{2}}{2}\left(\begin{array}{cccc}1 & 1 & 0 & 0 \\1 & -1 & 0 & 0 \\0 & 0 & 1 & 1 \\0 & 0 & 1 & -1\end{array}\right).$$

$\Inn_A$ is a Type 2 involution of $\Sp(4,\mathbb{Q})$ since $A^2 =I$ and each entry of $A$ is a $\mathbb{Q}$-multiple of $\sqrt{2}$. A basis for $E(A,1)$ that matches the conditions of Lemma \ref{Type2Class} is formed by the vectors $$v_1 =  \left(\begin{array}{c}0 \\ 0 \\ -\frac{1}{8} \\-\frac{1}{8}\end{array}\right)+ \sqrt{2} \left(\begin{array}{c} 0  \\ 0 \\-\frac{1}{8} \\0\end{array}\right)$$ and $$ v_2=\left(\begin{array}{c}  0 \\ 4 \\1 \\1\end{array}\right)+ \sqrt{2} \left(\begin{array}{c}2 \\ -2 \\1 \\0\end{array}\right).$$ 

It can be shown that $$v_3 =  \left(\begin{array}{c}0 \\ 0 \\ -\frac{1}{8} \\-\frac{1}{8}\end{array}\right)-\sqrt{2} \left(\begin{array}{c} 0  \\ 0 \\-\frac{1}{8} \\0\end{array}\right)$$ and $$ v_4=\left(\begin{array}{c}  0 \\ 4 \\1 \\1\end{array}\right)- \sqrt{2} \left(\begin{array}{c}2 \\ -2 \\1 \\0\end{array}\right)$$  are basis vectors for $E(A,-1)$ that also match the conditions of Lemma \ref{Type2Class}.

Following the notation of the previous lemma, we have $$X = \left(\begin{array}{cccc}0 & 0 & 0 & 2 \\0 & 4 & 0 & -2 \\-\frac{1}{8} & 1 & -\frac{1}{8} & 1 \\-\frac{1}{8} & 1 & 0 & 0\end{array}\right),$$ where $X^TJX = \left(\begin{smallmatrix}0 & \frac{1}{2} & 0 & 0 \\-\frac{1}{2} & 0 & 0 & 0 \\0 & 0 & 0 & \frac{1}{4} \\0 & 0 & -\frac{1}{4} & 0\end{smallmatrix}\right)$ and $A = \frac{\sqrt{2}}{2}X\left(\begin{smallmatrix}0 & I_{\frac{n}{2}} \\ 2I_{\frac{n}{2}} & 0\end{smallmatrix}\right) X^{-1}$ .
\end{beisp}

\begin{beisp}
Let $k$ be any field that does not contain $i = \sqrt{-1}$. For example, $k$ could be $\mathbb{R}$, or $\mathbb{F}_p$ or $\mathbb{Q}_p$ where $p$ is congruent to $3 \mod 4$. Consider the matrix $$A = i \left(\begin{array}{cccc}1 & 1 & 0 & 0 \\-2 & -1 & 0 & 0 \\0 & 0 & 1 & -2 \\0 & 0 & 1 & -1\end{array}\right).$$

$\Inn_A$ is a Type 2 involution of $\Sp(4,k)$ since $A^2 =I$ and each entry of $A$ is a $k$-multiple of $i$. A basis for $E(A,1)$ that matches the conditions of Lemma \ref{Type2Class} is formed by the vectors $$v_1 = \left(\begin{array}{c}-\frac{1}{2} \\1 \\1 \\1\end{array}\right)+i\left(\begin{array}{c}\frac{1}{2} \\0 \\-1 \\0\end{array}\right)$$ and $$ v_2=\left(\begin{array}{c}-\frac{1}{2} \\\frac{1}{2} \\1 \\1\end{array}\right)+i\left(\begin{array}{c}0 \\\frac{1}{2} \\-1 \\0\end{array}\right)$$ 

It can be shown that $$v_3 = \left(\begin{array}{c}-\frac{1}{2} \\1 \\1 \\1\end{array}\right)-i\left(\begin{array}{c}\frac{1}{2} \\0 \\-1 \\0\end{array}\right)$$ and $$ v_4=\left(\begin{array}{c}-\frac{1}{2} \\\frac{1}{2} \\1 \\1\end{array}\right)-i\left(\begin{array}{c}0 \\\frac{1}{2} \\-1 \\0\end{array}\right)$$  are basis vectors for $E(A,-1)$ that also match the conditions of Lemma \ref{Type2Class}.

Following the notation of the previous lemma, we have $$X = \left(\begin{array}{cccc}-\frac{1}{2} & -\frac{1}{2} & \frac{1}{2} & 0 \\1 & \frac{1}{2} & 0 & \frac{1}{2} \\1 & 1 & -1 & -1 \\1 & 1 & 0 & 0\end{array}\right),$$ where $X^TJX = \left(\begin{smallmatrix}0 & \frac{1}{2} & 0 & 0 \\-\frac{1}{2} & 0 & 0 & 0 \\0 & 0 & 0 & -\frac{1}{2} \\0 & 0 & \frac{1}{2} & 0\end{smallmatrix}\right)$ and $A = -iXJ X^{-1}$ .
\end{beisp}

Using our characterization of Type 2 involutions, we now find conditions on Type 2 involutions that are equivalent to isomorphy.

\begin{theorem}

Suppose $A$ and $B$ both induce Type 2 involutions of $\Sp(2n,k)$ where we write  $$A = \frac{\sqrt{\alpha}}{\alpha} X  \left(\begin{array}{cc}0 & I_{n} \\ \alpha I_{n} & 0\end{array}\right)  X^{-1}$$ and  $$B = \frac{\sqrt{\beta}}{\beta} Y  \left(\begin{array}{cc}0 & I_{n} \\ \beta I_{n} & 0\end{array}\right)  Y^{-1}$$ where $$X  = \left(\begin{array}{cccccccc}x_1 & x_2 & \cdots & x_{n} &y_1 & y_2 & \cdots & y_{n} \end{array}\right)\in \Gl(2n,k)$$ and $$Y  = \left(\begin{array}{cccccccc}\tilde{x_1} & \tilde{x_2} & \cdots & \tilde{x}_{n} &\tilde{y}_1 & \tilde{y}_2 & \cdots & \tilde{y}_{n} \end{array}\right)\in \Gl(2n,k),$$ where 
 for each $i$, we have that $x_i+\sqrt{\alpha}y_i \in E(A,1)$, $x_i-\sqrt{\alpha}y_i \in E(A,-1)$,  $\tilde{x}_i+\sqrt{\alpha}\tilde{y}_i \in E(B,1)$, $\tilde{x}_i-\sqrt{\alpha}\tilde{y}_i \in E(B,-1),$ and we know  that $$X^TJX = \frac{1}{2} \left(\begin{array}{cc}J & 0 \\0 & \frac{1}{\alpha} J\end{array}\right)$$ and $$ Y^TJY =  \frac{1}{2} \left(\begin{array}{cc}J & 0 \\0 & \frac{1}{\beta} J\end{array}\right).$$
 Then, $\Inn_A$ and $\Inn_B$ are isomorphic over $\Sp(2n,k)$ if and only $\alpha$ and $\beta$ lie in the same square class of $k$.

\end{theorem}

\begin{proof}

First, we note that if there exists $Q \in \Sp(2n, k)$ such that $Q^{-1}AQ = B$, then $\Inn_A$ and $\Inn_B$ are isomorphic over $\Sp(2n, k)$. Secondly, we note that this can be the case if and only if $\alpha$ and $\beta$ are in the same square class. So, to prove this theorem, we can simply assume that $\alpha = \beta$ and we will show that there exists such a a $Q \in \Sp(2n,k)$. 

Let $Q = XY^{-1}$. First, we note that $$Q^TJQ = (XY^{-1})^TJ(XY^{-1}) = (Y^{-1})^T (X^TJX)Y^{-1} = (Y^{-1})^T (Y^TJY)Y^{-1} = J,$$ so we see that $Q \in \Sp(2n,k)$. 

Lastly, we see that 
\begin{align*}
Q^{-1}AQ &= (XY^{-1})^{-1}A(XY^{-1})\\
&= Y(X^{-1}AX)Y^{-1}\\
&= \frac{\sqrt{\alpha}}{\alpha} Y  \left(\begin{array}{cc}0 & I_{n} \\ \alpha I_{n} & 0\end{array}\right)  Y^{-1}\\ 
&= B.
\end{align*}
\end{proof}

From here, it is clear that the number of Type 2 involution isomorphy classes is dependent on $n$ and on the number of square classes of the field $k$.

\begin{cor}
If $n$ is even, then $\Sp(2n,k)$ has at most $|k^*/(k^*)^2|-1$ isomorphy classes of Type 2 involutions. If $n$ is odd, then $\Sp(2n,k)$ has no Type 2 involutions. 
\end{cor}

\subsection{Type 3 Involutions}

We now examine the Type 3 case. Recall that $\phi$ is a Type 3 involution if $\phi = \Inn_A$, where $A \in \Sp(2n,k)$ and $A^2 = -I$. Such matrices have eigenvalues $\pm i$, and are diagonalizable because the minimal polynomial has no repeated roots. We begin by proving a couple of results about the eigenvectors of such matrices.

\begin{lem}

Suppose $A \in \Sp(2n,k) $ induces a Type 3 involution of $\Sp(2n,k)$. Also suppose $x,y \in k^n$ such that $x+iy \in E(A,-i)$. Then, $x-iy \in E(A,i)$. Likewise, if $u,v \in k^n$ such that $u+iv \in E(A,i)$. Then, $u-iv \in E(A,-i)$. Further, $\dim(E(A,i))= \dim(E(A,-i))$.

\end{lem}

\begin{proof}

Recall that complex conjugation preserves multiplication. This applies at the matrix level as well as at the scalar level. Because of this and since $$A(x+iy) = -i(x-iy) = y-ix,$$ then it follows that $$A(x-iy) = y+ix = i(x-iy).$$ That is, $x-iy \in E(A,-i)$. This proves the first statement. An analogous argument proves the second.

To see that $\dim(E(A,i))= \dim(E(A,-i))$ is the case, note that the first statement tells us that $\dim(E(A,i)) \le \dim(E(A,-i))$, and that the second statement tells us that $\dim(E(A,i))\ge \dim(E(A,-i))$.
\end{proof}

\begin{lem}
\label{Type3Eigen}
Suppose $A \in \Sp(2n,k) $ induces a Type 3 involution of $\Sp(2n,k)$. Then, there exists $x_1,...,x_n, y_1,...,y_n \in k^{2n}$ such that the $x_j+iy_j$ form a basis for  $E(A,i)$ and the $x_j-iy_j$ form a basis for  $E(A,-i)$.
\end{lem}

\begin{proof}

Since $\Inn_A$ is Type 3, then we are assuming that $A \in \Sp(2n,k)$ and $A^2 = -I$. It follows that all eigenvalues of $A$ are $\pm i$. Since there are no repeated roots in the minimal polynomial of $A$, then we see that $A$ is diagonalizable. We wish to construct bases for $E(A,i)$ and $E(A,-i)$ such that all the vectors lie in $k[i]^{2n}$.  Let $\{z_1,...,z_{2n}\}$ be a basis for $k^{2n}$. For each $j$, let $u_j = (A+iI)z_j.$ Note that $$Au_j = A(A+iI)z_j = (A^2+iA)z_j = (-I+iA)z_j= i(A+iI)z_j= iu_j.$$ So, $\{u_1,...,u_{2n}\}$ must span $E(A,i)$. Thus, we can appropriately choose $n$ of these vectors and form a basis for $E(A,i)$. We can reorder, and assume that the $n$ chosen vectors are $u_1,...,u_n$. Let $x_j = Ax_j$ and $y_j = z_j$. Then, these eigenvectors are of the form $x_j+iy_j$. By the previous lemma, we know that $x_j-iy_j \in E(A,-i)$. This proves the statement.

\end{proof}

We are now able to prove results that characterize the matrices that induce Type 3 involutions, and then use these characterizations to find conditions on these involutions that are equivalent to isomorphy. We will have to prove this by looking at separate cases depending on whether or not $i = \sqrt{-1}$ lies in $k$. We begin by assuming that $i \in k$.

\begin{lem}
\label{Type3ClassYes}
Assume $i \in k$ and suppose $\theta = \Inn_A$ is a Type 3 involution of $\Sp(2n,k)$, where $A \in \Sp(2n,k)$. Then, $A = X \left(\begin{smallmatrix}iI_{n} &0 \\ 0& -iI_{n}\end{smallmatrix}\right) X^{-1}$ for some $X \in \Gl(n,k),$ where $X^TJX = \left(\begin{smallmatrix}0 & X_1 \\-X_1 & 0\end{smallmatrix}\right)$ where $X_1$ is diagonal.
\end{lem}

\begin{proof}

We know from Lemma \ref{Type3Eigen} that we have bases for $E(A,i)$ and $E(A,-i)$ that lie in $k^{2n}$. We will show that we can in fact choose bases $a_1,...,a_{n}$ for $E(A,i) \cap k^{2n}$ and $b_1,...,b_{n}$ for $E(A,-i) \cap k^{2n}$ such that $\beta(a_j,a_l) = 0 = \beta(b_j,b_l)$ and $\beta(a_j, b_l)$ is nonzero if and only if $j=l$. We will build these bases recursively.

First, we know that we can choose some nonzero $a_1 \in E(A,i) \cap k^{2n}$. Then, since $\beta$ is non degenerate, we can choose a vector $t$ such that $\beta(a_1, t) \ne 0$. We note that $E(A,i) \oplus E(A,-i) = k^{2n}$, so we can choose $t_{i} \in E(A,i)\cap k^{2n}$ and $t_{-i} \in E(A,-i)\cap k^{2n}$ such that $t = t_{i}+t_{-i}$. Since $\beta(a_1, t_{i}) = 0$, then it follows that $\beta(a_1, t_{-i}) \in k$ is nonzero. Let $b_1 = t_{-i}$.

Let $E_1 = \Span_k(a_1,b_1)$ and let $F_1$ be the orthogonal complement of $E_1$ in $k^{2n}$. Since the system of linear equations $$\beta(a_1,x) = 0$$ $$\beta(b_1,x) =0$$ has $2n-2$ free variables, then we see that $F_1$ has dimension $2n-2$. 

We now wish to find $a_2 \in F_1 \cap E(A,i)$. Similar to the construction in the previous lemma, we can choose $x \in F_1$, and let $a_2 = Ax+ix$. Now we want $b_2 \in F_2 \cap E(A-,i)$ such that $\beta(a_2, b_2) $ is nonzero. Since $\beta|_{F_1}$ is non degenerate, then there exists some $y \in F_2$ such that $\beta(a_2,y) \ne 0$. Similar to the construction of $b_1$, we see that this implies the existence of a vector $b_2$ that fits our criteria. 

Now, we let $E_2 = \Span_k(a_1,a_2,b_1,b_2)$ and let $F_2$ be the orthogonal complement of $E_2$ in $k^n$. We continue this same argument $n$ times, until we have the bases that we wanted to find. Let $$X= (a_1,...,a_{n}, b_1,...,b_{n}).$$ Then, the result follows.
\end{proof}

We can now use this characterization to show that all such involutions must be isomorphic.

\begin{theorem}
\label{type3lemYes}
Assume that $i \in k$. Then, if $\Inn_A$ and $\Inn_B$ are both Type 3 involutions of $\Sp(2n,k)$, then $\Inn_A$ and $\Inn_B$ are isomorphic over $\Sp(2n,k)$.
\end{theorem}

\begin{proof}

Suppose we have two such involutions of $\Sp(2n,k)$. Let them be represented by matrices $A,B \in \Sp(2n,k)$. By the previous Lemma, we can choose $X, Y \in \Gl(n,k)$ such that $$X^{-1}AX = \left(\begin{array}{cc}iI_n & 0 \\0 & -iI_n \end{array}\right) = Y^{-1}BY,$$ $$X^TJX = \left(\begin{array}{cc}0 & X_1 \\-X_1 & 0\end{array}\right)$$ and 
 $$Y^TJY = \left(\begin{array}{cc}0 & Y_1 \\-Y_1 & 0\end{array}\right)$$ where $X_1$ and $Y_1$ are diagonal.
 
Since $X_1$ and $Y_1$ are both invertible diagonal matrices, then we can choose $R_1$ and $R_2 \in \Gl(\frac{n}{2},k)$ such that $Y_1 = R_1^TX_1R_2$. Let $R = \left(\begin{smallmatrix}R_1 & 0 \\0 & R_2\end{smallmatrix}\right)$ and $Q = XRY^{-1}$. We will show that $Q \in \Sp(2n,k)$ and $Q^{-1}AQ = B$. This will then prove that $\Inn_A$ and $\Inn_B$ lie in the same isomorphy class.

First we show that $Q \in \Sp(2n,k)$. Note that 
\begin{align*}
Q^TJQ &= (XRY^{-1})^TJ(XRY^{-1})\\ 
&= (Y^{-1})^TR^T(X^TJX)RY^{-1}\\
&= (Y^{-1})^T(Y^TJY)Y^{-1}\\ 
&= J,
\end{align*}
which proves this claim. 

Lastly, we show that $Q^{-1}AQ = B$. We first note that $R$ and $\left(\begin{smallmatrix}-iI & 0 \\0 & iI\end{smallmatrix}\right)$ commute. Then, we see that
\begin{align*}
Q^{-1}AQ &= (XRY^{-1})^{-1}A (XRY^{-1})\\ 
&= YR^{-1}(X^{-1}AX)RY^{-1}\\
&= Y R^{-1}\left(\begin{array}{cc}-iI & 0 \\0 & iI\end{array}\right)R Y^{-1}\\ 
&= Y R^{-1}R\left(\begin{array}{cc}-iI & 0 \\0 & iI\end{array}\right) Y^{-1}\\
&= Y \left(\begin{array}{cc}-iI & 0 \\0 & iI\end{array}\right) Y^{-1}\\
&= B.
\end{align*}
We have shown what was needed.
\end{proof}

We now examine the case where $i \not \in k$, beginning with a characterization of the matrices that induce these involutions.

\begin{lem}
\label{Type3ClassNo}
Assume $i \not \in k$. Suppose $\theta = \Inn_A$ is a Type 3 involution of $\Sp(2n,k)$. Then, $A = U \left(\begin{smallmatrix}0 & I_{n} \\ -I_{n} & 0\end{smallmatrix}\right) U^{-1} = UJU^{-1}$ for $$U  = \left(\begin{array}{cccccccccc}a_1 & a_2 & \cdots & a_n &b_1 & b_2 & \cdots & b_n  \end{array}\right)\in \Gl(2n,k),$$ where the $a_j+ib_j$ are a basis for $E(A,i)$, the $a_j-ib_j$ are a basis for $E(A,-i)$, and $U^TJU = \left(\begin{smallmatrix}0 & U_1 \\-U_1 & 0\end{smallmatrix}\right)$, where $U_1$ is diagonal.
\end{lem}

\begin{proof}

We know from Lemma \ref{Type3Eigen} that we have bases for $E(A,i)$ and $E(A,-i)$ that lie in $k[i]^{2n}$. We will show that we can in fact choose bases $a_1+ib_1,...,a_{n}+ib_{n}$ for $E(A,i) \cap k[i]^{2n}$ and $a_1-ib_1,...,a_{n}-ib_{n}$ for $E(A,-i) \cap k[i]^{2n}$ such that $\beta(a_j+ib_j, a_l-ib_l)$ is nonzero if and only if $j=l$. From this, we will be able to show that $\beta(a_j,a_l) = 0 = \beta(b_j,b_l)$ when $j \ne l$ and $\beta(a_j, b_l) = 0$ for all $j$ and $l$. We will build these bases recursively.

Recall that given any vector $x \in k^{2n}$, we know that $Ax+ix \in E(A,i)$. We want to choose $x\in k^{2n}$ such that $x^TA^TJx \ne 0$. (The reasons for this will become apparent.) If $e_j^TA^TJe_j \ne 0$, we can let $x=e_j$. Suppose that this doesn't occur for any $j$.

 Since $A^TJ$ is invertible, we know that for more than $2n$ pairs of $j$ and $l$ we have  $e_j^TA^TJe_l \ne 0$. Also, we see that since $A$ is symplectic and $A^TJA = J$, then we have that 
$$A^TJ = JA^{-1} = JA^3 = -JA$$ and that $$(A^TJ)^T = J^TA = -JA = A^TJ.$$
That is, $A^TJ$ is symmetric. So, $e_j^TA^TJe_j = e_l^TA^TJe_l $.
Then, we can let $x = e_j+e_l$. Then, we have 
$$x^TA^TJx = e_jA^TJe_l+e_lA^TJe_j = 2e_jA^TJe_l \ne 0.$$ In either case, we have many choices for $x$.

Let $x \in k^{2n}$ be a vector from above. We have $Ax+ix \in E(A,i)$. Let $a_1 = Ax$ and $b_1 = x$. So, $a_1+ib_1 \in E(A,i)$ and $a_1-ib_1 \in E(A,-i)$.  From this, it follows that 
\begin{align*}
\beta(a_1+ib_1, a_1-ib_1) &= (\beta(a_1,a_1)+\beta(b_1,b_1))+i(-\beta(a_1,b_1)+\beta(b_1,a_2))\\
&=0+i(-\beta(Ax,x)+\beta(x,Ax)\\
&=-2i\beta (Ax,x)\\  
&= -2i(x^TA^TJx)\\
&\ne 0.
\end{align*}

Let $E_1 = \Span_{k[i]}(a_1+ib_1, a_1-ib_1) = \Span_{k[i]}(a_1,b_1)$, and let $F_1$ be the orthogonal complement of $E_1$ over $k[i]$. $F_1$ has dimension $2n-2$, and $\beta|_{F_1}$ is nondegenerate. So, we can find a nonzero vector $x \in F_1 \cap k^{2n}$ such that $\beta|_{F_1}(Ax,x) \ne 0$. So, as in the last case, let $a_2 = Ax$ and $b_2 = x$. Similar to before, we have $\beta(a_2+ib_2, a_2-ib_2) \ne 0$.

Let $E_2 = \Span_{k[i]}(a_1,a_2,b_1,b_2)$, and let $F_2$ be the orthogonal complement of $E_2$ over $k[i]$. In this manner, we can create the bases that we noted in the opening paragraph of this proof. 

Note that we always have $$0 = \beta(a_j+ib_j, a_l+ib_l) = (\beta(a_j,a_l)-\beta(b_j,b_l))+i(\beta(a_j,b_l)+\beta(b_j,a_l)),$$ and when $j \ne l$ we have $$0 = \beta(a_j+ib_j, a_l-ib_l) = (\beta(a_j,a_l)+\beta(b_j,b_l))+i(-\beta(a_j,b_l)+\beta(b_j,a_l)).$$

This tells us that when $j \ne l$ that $$\beta(a_j,b_l) = \beta(a_j,a_l) = \beta(b_j,b_l) = 0.$$ 

When $j = l$, we know that $\beta(b_j,b_j) =0= \beta(a_j,a_j)$. Lastly, we see that $\beta(a_j, b_j)= -\beta(b_j,a_j)$. 

Let $$U = (a_1,...,a_{n},b_1,...,b_{n}).$$ Then, it follows that $U^TJU = \left(\begin{smallmatrix}0 & U_1 \\-U_1 & 0\end{smallmatrix}\right)$ where $X_1$ is a diagonal $n \times n$ matrix.

Lastly, since $Ab_j=a_j$, then it follows that $Aa_j = -b_j$. So, we have that  $$A = U \left(\begin{array}{cc}0 & I_{n} \\ -I_{n} & 0\end{array}\right) U^{-1}.$$

\end{proof}

We now show that if $i \not \in k$, then we also have that there is only one isomorphy class of Type 3 involutions.

\begin{theorem}
\label{type3lemYes}
Assume $i \not \in k$. Then, if $\Inn_A$ and $\Inn_B$ are both Type 3 involutions of $\Sp(2n,k)$, then $\Inn_A$ and $\Inn_B$ are isomorphic over $\Sp(2n,k)$.
\end{theorem}

\begin{proof}

By the previous Lemma, we can choose a matrix $U \in \Gl(n,k)$ such that $$A = U \left(\begin{array}{cc}0 & -I_{\frac{n}{2}} \\ I_{\frac{n}{2}} & 0\end{array}\right) U^{-1}$$ for $$U  = \left(\begin{array}{cccccccccc}a_1 & a_2 & \cdots & a_\frac{n}{2} &b_1 & b_2 & \cdots & b_\frac{n}{2}  \end{array}\right)\in \Gl(n,k),$$ where the $a_j+ib_j$ are a basis for $E(A,i)$, the $a_j-ib_j$ are a basis for $E(A,-i)$, and $U^TJU = \left(\begin{smallmatrix}0 & U_1 \\ -U_1 & 0\end{smallmatrix}\right)$ for diagonal matrix $U_1$.

Let $$X = (a_1+ib_1,...,a_{\frac{n}{2}}+ib_{\frac{n}{2}}, a_1-ib_1,...,a_{\frac{n}{2}}-ib_{\frac{n}{2}}),$$ and consider $\Inn_A$ and $\Inn_B$ as involutions of $\Sp(2n,k[i])$. By construction, we see that $X$ is a matrix that satisfies the conditions of Lemma \ref{Type3ClassYes} for the group $\Sp(2n,k[i])$. We note that $X_1 = -2iU_1$. We also know by the previous Theorem that $\Inn_A$ and $\Inn_B$ are isomorphic over $\Sp(2n,k[i])$. So, we can choose $Q_i \in \Sp(2n,k[i])$ such that $Q_i^{-1}AQ_i = B$. Let $Y = Q_i^{-1}X$. We now show a couple of facts about $Y$.

First, we note that since $Y$ was obtained from $X$ via row operations, then for $1 \le j \le \frac{n}{2}$, the $j$th and $\frac{n}{2}+j$th columns are $i$-conjugates of one another.

Also, note that 
\begin{align*}
Y^{-1}BY &= (Q_i^{-1}X)^{-1}B (Q_i^{-1}X)\\
&= X^{-1}Q_iBQ_i^{-1}X\\ 
&= X^{-1}AX\\ 
&= \left(\begin{array}{cc} -iI_{\frac{n}{2}} & 0 \\0&  iI_{\frac{n}{2}} \end{array}\right).
\end{align*}

Lastly, we see that 
\begin{align*}
Y^TJY &= (Q_i^{-1}X)^TJ(Q_i^{-1}X)\\
&= X^T((Q_i^{-1})^TJQ_i)X\\
&= X^TJX\\
&= \left(\begin{array}{cc} 0 & X_1 \\ -X_1 & 0 \end{array}\right)\\ 
&= \left(\begin{array}{cc} 0 & -2iU_1 \\ 2iU_1 & 0 \end{array}\right).
\end{align*}

Write $$Y = (c_1+id_1,...,c_{\frac{n}{2}}+id_{\frac{n}{2}}, c_1-id_1,...,c_{\frac{n}{2}}-id_{\frac{n}{2}}),$$ and let $$V = (c_1,...,c_{\frac{n}{2}},d_1,...,d_{\frac{n}{2}}).$$ It follows from what we have shown that $$B = V \left(\begin{array}{cc}0 & -I_{\frac{n}{2}} \\ I_{\frac{n}{2}} & 0\end{array}\right) V^{-1} \text{ where } V^TJV = \left(\begin{array}{cc}0 & U_1 \\ -U_1 & 0\end{array}\right) = U^TMU.$$

Now, let $Q = UV^{-1}$. We will show that $Q^{-1}AQ = B$ and $Q \in \Sp(2n,k)$. This will prove that $\Inn_A$ and $\Inn_B$ are isomorphic over $\Sp(2n,k)$. 

We first show that $Q \in \Sp(2n,k)$.  
\begin{align*}
Q^TJQ &= (UV^{-1})^TJUV^{-1}\\ 
&= (V^{-1})^T(U^TJU)V^{-1}\\ 
&= (V^{-1})^T(V^TJV)V^{-1}\\ 
&= J.
\end{align*}

Lastly, we show that $Q^{-1}AQ = B$. 
\begin{align*}
Q^{-1}AQ &= (UV^{-1})^{-1}A(UV^{-1})\\ 
&= VU^{-1}AUV^{-1}\\
&= V\left(\begin{array}{cc}0 & -I_{\frac{n}{2}} \\ I_{\frac{n}{2}} & 0\end{array}\right) V^{-1}\\ 
&= B.
\end{align*}
We have shown what was needed.

\end{proof}

Combining the results from this section, we get the following corollary. 

\begin{cor}
\label{CorType3}
If $\Inn_A$ and $\Inn_B$ are both Type 3 involutions of $\Sp(2n,k)$, then $\Inn_A$ and $\Inn_B$ are isomorphic over $\Sp(2n,k)$. That is, $\Sp(2n,k)$ has exactly one isomorphy class of Type 3 involutions. Further, the matrix $J$ is a representative matrix for this isomorphy class.
\end{cor}

\subsection{Type 4 Involutions}

We now move on to a similar classification in the Type 4 case. First, we characterize the eigenvectors of the matrices that induce these involutions. Recall that we can choose $A \in \Sp(2n,k[\sqrt{\alpha}])$ such that each entry of $A$ is a $k$-multiple of $\sqrt{\alpha}$, and that we know $A^2 = -I$. We begin by proving a couple of lemmas about the eigenspaces of these matrices.

\begin{lem}

Suppose $A \in \Sp(2n,k[\sqrt{\alpha}]) $ induces a Type 4 involution of $\Sp(2n,k)$. Also suppose $x,y \in k^{2n}$ such that $x+\sqrt{-\alpha}y \in E(A,i)$. Then, $x-\sqrt{-\alpha}y \in E(A,-i)$. Likewise, if $u,v \in k^{2n}$ such that $u+\sqrt{-\alpha}v \in E(A,-i)$. Then, $u-\sqrt{-\alpha}v \in E(A,i)$. Further, $\dim(E(A,i))= \dim(E(A,-i))$.

\end{lem}

\begin{proof}

Suppose $x,y \in k^n$ such that $x+\sqrt{-\alpha}y \in E(A,-i)$. Then,
$$A(x+\sqrt{-\alpha}y) = -i(x+\sqrt{-\alpha}y)$$ which implies
$$Ax+\sqrt{-\alpha}Ay = \sqrt{\alpha}y-ix.$$
Then, complex conjugation tells us that 
$$Ax-\sqrt{-\alpha}Ay = \sqrt{\alpha}y+ix,$$ which tells us that
$$A(x-\sqrt{-\alpha}y) = i(x-\sqrt{-\alpha}y).$$ A similar argument shows that if $u,v \in k^n$ such that $u+\sqrt{-\alpha}v \in E(A,i)$. Then, $u-\sqrt{-\alpha}v \in E(A,-i)$.

Since  $x+\sqrt{-\alpha}y \in E(A,-i)$ implies $x-\sqrt{-\alpha}y \in E(A,i)$ and vice versa, then we see that $\dim(E(A,i))= \dim(E(A,-i))$.

\end{proof}

\begin{lem}
\label{Type4Eigen}
Suppose $\theta = \Inn_A$ is a Type 4 involution of $\Sp(2n,k)$ where $A \in \Sp(2n,k[\sqrt{\alpha}])$. Then, we can find $x_1,...,x_{n}, y_1,...,y_{n} \in k^n$ such that the $x+\sqrt{-\alpha}y$ are a basis for $E(A,i)$ and the $x-\sqrt{-\alpha}y$ are a basis for $E(A,-i)$.
\end{lem}

\begin{proof}

Since $\Inn_A$ is Type 4, then we are assuming that $A \in \Sp(2n,k[\sqrt{\alpha}])$ and $A^2 = -I$. It follows that all eigenvalues of $A$ are $\pm i$. Since there are no repeated roots in the minimal polynomial of $A$, then we see that $A$ is diagonalizable. We wish to construct bases for $E(A,i)$ and $E(A,-i)$ such that all the vectors lie in $k[i]^{2n}$.  Let $\{z_1,...,z_{2n}\}$ be a basis for $k^{2n}$. For each $j$, let $u_j = (\sqrt{\alpha}A+\sqrt{-\alpha}I)z_j.$ Note that $$Au_j = A(\sqrt{\alpha}A+\sqrt{-\alpha}I)z_j = (\sqrt{\alpha}A^2+\sqrt{-\alpha}A)z_j = i(\sqrt{\alpha}A+\sqrt{-\alpha}I)z_j= iu_j.$$ So, $\{u_1,...,u_n\}$ must span $E(A,i)$. Thus, we can appropriately choose $n$ of these vectors and form a basis for $E(A,i)$. Note that each of these vectors lies in $k[i]^{2n}$. Label these basis vectors as $v_1,...,v_n$. We can write each of these vectors as $v_j = x_j+\sqrt{-\alpha}y_j$. By the previous lemma, we know that $x_j-\sqrt{-\alpha}y_j \in E(A,-i)$, and that these vectors form a basis for $E(A,-i)$.

\end{proof}

We are now able to prove results that characterize the matrices that induce Type 4 involutions, and then use these characterizations to find conditions on these involutions that are equivalent to isomorphy. We will have separate cases, depending on whether or not $\sqrt{-\alpha}$ lies in $k$. We begin by assuming that $\sqrt{-\alpha} \in k$. Since we are also assuming that $\sqrt{\alpha} \not \in k$, then it follows from these two assumptions that $\alpha$ and $-1$ lie in the same square class of $k$. Thus, we can assume in this case that $\alpha = -1$, which means $\sqrt{-\alpha} = 1$.

\begin{lem}
\label{Type4ClassYes}
Assume $\sqrt{-\alpha} \in k$ and suppose $\theta$ is a Type 4 involution of $\Sp(2n,k)$. Then, $A = X \left(\begin{smallmatrix}iI_{\frac{n}{2}} &0 \\ 0& -iI_{\frac{n}{2}}\end{smallmatrix}\right) X^{-1}$ for some $X \in \Gl(2n,k),$ where $X^TJX =  \left(\begin{smallmatrix}0 & X_1 \\ -X_1& 0\end{smallmatrix}\right)$ and $X_1$ is diagonal.
\end{lem}

\begin{proof}

We know from Lemma \ref{Type4Eigen} that we have bases for $E(A,i)$ and $E(A,-i)$ that lie in $k^{2n}$. We will show that we can in fact choose bases $a_1,...,a_{n}$ for $E(A,i) \cap k^{2n}$ and $b_1,...,b_{n}$ for $E(A,-i) \cap k^{2n}$ such that $\beta(a_j,a_l) = 0 = \beta(b_j,b_l)$ and $\beta(a_j, b_l)$ is nonzero if and only if $j=l$. We will build these bases recursively.

First, we know that we can choose some nonzero $a_1 \in E(A,i) \cap k^{2n}$. Then, since $\beta$ is non degenerate, we can choose a vector $t$ such that $\beta(a_1, t) \ne 0$. We note that $E(A,i) \oplus E(A,-i) = k^{2n}$, so we can choose $t_{i} \in E(A,i)\cap k^{2n}$ and $t_{-i} \in E(A,-i)\cap k^{2n}$ such that $t = t_{i}+t_{-i}$. Since $\beta(a_1, t_{i}) = 0$, then it follows that $\beta(a_1, t_{-i}) \in k$ is nonzero. Let $b_1 =  t_{-i}.$

Let $E_1 = \Span_k(a_1,b_1)$ and let $F_1$ be the orthogonal complement of $E_1$ in $k^{2n}$. Since the system of linear equations $$\beta(a_1,x) = 0$$ $$\beta(b_1,x) =0$$ has $2n-2$ free variables, then we see that $F_1$ has dimension $2n-2$. 

We now wish to find $a_2 \in F_2 \cap E(A,i)$. Similar to the construction in the previous lemma, we can choose $x \in F_1$, and let $a_2 = \sqrt{\alpha}Ax+\sqrt{-\alpha}x$. Now we want $b_2 \in F_2 \cap E(A-,i)$ such that $\beta(a_2, b_2)$ is nonzero. Since $\beta|_{F_1}$ is non degenerate, then there exists some $y \in F_2$ such that $\beta(a_2,y) \ne 0$. Similar to the construction of $b_1$, we see that this implies the existence a vector $b_2$ that fits our criteria. 

Now, we let $E_2 = \Span_k(a_1,a_2,b_1,b_2)$ and let $F_2$ be the orthogonal complement of $E_2$ in $k^n$. We continue this same argument $n$ times, until we have the bases that we wanted to find. Let $$X= (a_1,...,a_{n}, b_1,...,b_{n}).$$ Then, the result follows.
\end{proof}

Here is an example of a Type 4 involution when $\sqrt{-\alpha} \in k$.

\begin{beisp}
Let $k$ be $\mathbb{R}$. So, $\alpha = -1$. Notice that $\sqrt{-\alpha} =1 \in \mathbb{R}$. Consider the matrix $$A = i \left(\begin{array}{cccc}0 & 1 & 0 & 0 \\1 & 0 & 0 & 0 \\0 & 0 & 0 & -1 \\0 & 0 & -1 & 0\end{array}\right).$$

$\Inn_A$ is a Type 4 involution of $\Sp(4,k)$ since $A^2 =-I$ and each entry of $A$ is a $k$-multiple of $i$. A basis for $E(A,1)$ that matches the conditions of Lemma \ref{Type4ClassYes} is formed by the vectors $v_1 = \frac{\sqrt{2}}{2} \left(\begin{smallmatrix}1 \\1 \\0 \\0\end{smallmatrix}\right)$ and $ v_2=  \frac{\sqrt{2}}{2} \left(\begin{smallmatrix}0 \\0 \\-1 \\1\end{smallmatrix}\right).$  It can also be shown that $v_3 =   \frac{\sqrt{2}}{2} \left(\begin{smallmatrix}0 \\0 \\1 \\1\end{smallmatrix}\right)$ and $ v_4=  \frac{\sqrt{2}}{2} \left(\begin{smallmatrix}1 \\-1 \\0 \\0\end{smallmatrix}\right)$ are basis vectors for $E(A,-1)$ that also match the conditions of Lemma \ref{Type4ClassYes}.

Following the notation of the previous lemma, we have $$X = \frac{\sqrt{2}}{2} \left(\begin{array}{cccc}1 & 0 & 0 & 1 \\1 & 0 & 0 & -1 \\0 & -1 & 1 & 0 \\0 & 1 & 1 & 0\end{array}\right),$$ where $X^TJX = J$ and $A = X \left(\begin{smallmatrix}iI_{\frac{n}{2}} &0 \\ 0& -iI_{\frac{n}{2}}\end{smallmatrix}\right) X^{-1}$ .
\end{beisp}

Now we characterize the isomorphy classes of Type 4 involutions in the case where $\sqrt{-\alpha} \in k$.

\begin{theorem}
\label{type4lemYes}
Assume that $\sqrt{-\alpha} \in k$. Then, if $\Inn_A$ and $\Inn_B$ are both Type 4 involutions of $\Sp(2n,k)$ such that $A, B \in \Sp(2n, k[\sqrt{\alpha}])$, then $\Inn_A$ and $\Inn_B$ are isomorphic over $\Sp(2n,k)$.
\end{theorem}

\begin{proof}

Suppose we have two such involutions of $\Sp(2n,k)$. Let them be represented by matrices $A,B \in \Sp(2n,k)$. By the previous Lemma, we can choose $X, Y \in \Gl(n,k)$ such that $$X^{-1}AX = \left(\begin{array}{cc}iI_n & 0 \\0 & -iI_n\end{array}\right) = Y^{-1}BY,$$ $$X^TJX = \left(\begin{array}{cc}0 & X_1 \\-X_1 & 0\end{array}\right),$$ and $$  Y^TJY\left(\begin{array}{cc}0 & Y_1 \\-Y_1 & 0\end{array}\right),$$ where $X_1$ and $Y_1$ are diagonal.

Since $X_1$ and $Y_1$ are both invertible diagonal matrices, then we can choose $R_1$ and $R_2 \in \Gl(\frac{n}{2},k)$ such that $Y_1 = R_1^TX_1R_2$. Let $R = \left(\begin{smallmatrix}R_1 & 0 \\0 & R_2\end{smallmatrix}\right)$ and $Q = XRY^{-1}$. It follows from this that $R^TX^TJXR = Y^TJY$. We will show that $Q \in \Sp(2n,k)$ and $Q^{-1}AQ = B$. This will then prove that $\Inn_A$ and $\Inn_B$ lie in the same isomorphy class.

First we show that $Q \in \Sp(2n,k)$. Note that $$Q^TJQ = (XRY^{-1})^TJ(XRY^{-1}) = (Y^{-1})^TR^T(X^TJX)RY^{-1}$$ $$ = (Y^{-1})^T(Y^TJY)Y^{-1} = J,$$ which proves this claim. 

Lastly, we show that $Q^{-1}AQ = B$. We first note that $R$ and $\left(\begin{smallmatrix}iI & 0 \\0 & -iI\end{smallmatrix}\right)$ commute. Then, we see that
\begin{align*}
Q^{-1}AQ &= (XRY^{-1})^{-1}A (XRY^{-1})\\ 
&= YR^{-1}(X^{-1}AX)RY^{-1}\\
&= Y R^{-1}\left(\begin{array}{cc}-iI & 0 \\0 & iI\end{array}\right)R Y^{-1}\\ 
&= Y R^{-1}R\left(\begin{array}{cc}-iI & 0 \\0 & iI\end{array}\right) Y^{-1}\\
&= Y \left(\begin{array}{cc}-iI & 0 \\0 & iI\end{array}\right) Y^{-1}\\
&= B.
\end{align*}
\end{proof}

We now examine the case where $\sqrt{-\alpha} \not \in k$. We begin with a characterization of the matrices that induce Type 4 involutions in this case.

\begin{lem}
\label{Type4ClassNo}
Assume $\sqrt{-\alpha} \not \in k$. Suppose $\theta = \Inn_A$ is a Type 4 involution of $\Sp(2n,k)$. Then, $A = \frac{\sqrt{\alpha}}{\alpha}U \left(\begin{smallmatrix}0 & I_{n} \\ -\alpha I_{n} & 0\end{smallmatrix}\right) U^{-1} $ for $$U  = \left(\begin{array}{cccccccccc}a_1 & a_2 & \cdots & a_n &b_1 & b_2 & \cdots & b_n  \end{array}\right)\in \Gl(2n,k),$$ where the $a_j+\sqrt{-\alpha}b_j$ are a basis for $E(A,i)$, the $a_j-\sqrt{-\alpha}b_j$ are a basis for $E(A,-i)$, and $U^TJU = \left(\begin{smallmatrix}0 & U_1 \\-U_1 & 0\end{smallmatrix}\right)$, where $U_1$ is diagonal.
\end{lem}

\begin{proof}

We know from Lemma \ref{Type4Eigen} that we have bases for $E(A,i)$ and $E(A,-i)$ that lie in $k[\sqrt{-\alpha}]^{2n}$. We will show that we can in fact choose bases $a_1+\sqrt{-\alpha}b_1,...,a_{n}+\sqrt{-\alpha}b_{n}$ for $E(A,i) \cap k[\sqrt{-\alpha}]^{2n}$ and $a_1-\sqrt{-\alpha}b_1,...,a_{n}-\sqrt{-\alpha}b_{n}$ for $E(A,-i) \cap k[\sqrt{-\alpha}]^{2n}$ such that $\beta(a_j+ib_j, a_l-ib_l)$ is nonzero if and only if $j=l$. From this, we will be able to show that $\beta(a_j,a_l) = 0 = \beta(b_j,b_l)$ when $j \ne l$ and $\beta(a_j, b_l) = 0$ for all $j$ and $l$. We will build these bases recursively.

Recall that given any vector $x \in k^{2n}$, we know that $\sqrt{\alpha}Ax+\sqrt{-\alpha}x \in E(A,i)$. We want to choose $x\in k^{2n}$ such that $x^TA^TJx \ne 0$. That is, such that $\beta(Ax,x) \ne 0$. (The reasons for this will become apparent.) If $e_j^TA^TJe_j \ne 0$, we can let $x=e_j$. Suppose that this doesn't occur for any $j$.

 Since $A^TJ$ is invertible, we know that for more than $2n$ pairs of $j$ and $l$ we have  $e_j^TA^TJe_l \ne 0$. Also, we see that since $A$ is symplectic and $A^TJA = J$, then we have that 
$$A^TJ = JA^{-1} = JA^3 = -JA$$ and that $$(A^TJ)^T = J^TA = -JA = A^TJ.$$
That is, $A^TJ$ is symmetric. So, $e_j^TA^TJe_j = e_l^TA^TJe_l $.
Then, we can let $x = e_j+e_l$. Then, we have 
$$x^TA^TJx = e_jA^TJe_l+e_lA^TJe_j = 2e_jA^TJe_l \ne 0.$$ In either case, we have many choices for $x$.

Let $x \in k^{2n}$ be a vector from above. We have $\sqrt{\alpha}Ax+\sqrt{-\alpha}x \in E(A,i)$. Let $a_1 = \sqrt{\alpha}Ax$ and $b_1 = x$. So, $a_1+\sqrt{-\alpha}b_1 \in E(A,i)$ and $a_1-\sqrt{-\alpha}b_1 \in E(A,-i)$.  From this, it follows that 
\begin{align*}
\beta(a_1+\sqrt{-\alpha}b_1, a_1-\sqrt{-\alpha}b_1) &= (\beta(\sqrt{\alpha}Ax,\sqrt{\alpha}Ax) +\alpha\beta(x,x))+\sqrt{-\alpha}(-\beta(\sqrt{\alpha}Ax,x)+\beta(x,\sqrt{\alpha}Ax)\\
&=2\alpha i\beta (x,Ax)\\
&\ne 0.
\end{align*}

Let $E_1 = \Span_{k[\sqrt{-\alpha}]}(a_1+\sqrt{-\alpha}b_1, a_1-\sqrt{-\alpha}b_1) = \Span_{k[\sqrt{-\alpha}]}(a_1,b_1)$, and let $F_1$ be the orthogonal complement of $E_1$ over $k[\sqrt{-\alpha}]$. $F_1$ has dimension $2n-2$, and $\beta|_{F_1}$ is nondegenerate. So, we can find a nonzero vector $x \in F_1 \cap k^{2n}$ such that $\beta|_{F_1}(x,-Ax) \ne 0$. So, as in the last case, let $a_2 = \sqrt{\alpha}Ax$ and $b_2 = x$. As before, we have $\beta(a_2+\sqrt{-\alpha}b_2, a_2-\sqrt{-\alpha}b_2) \ne 0$.

Let $E_2 = \Span_{k[\sqrt{-\alpha}]}(a_1,a_2,b_1,b_2)$, and let $F_2$ be the orthogonal complement of $E_2$ over $k[\sqrt{-\alpha}]$. In this manner, we can create the bases that we noted in the opening paragraph of this proof. 

Note that we always have $$0 = \beta(a_j+\sqrt{-\alpha}b_j, a_l+\sqrt{-\alpha}b_l) = (\beta(a_j,a_l)-\alpha \beta(b_j,b_l))+\sqrt{-\alpha}(\beta(a_j,b_l)+\beta(b_j,a_l)),$$ and when $j \ne l$ we have $$0 = \beta(a_j+\sqrt{-\alpha}b_j, a_l-\sqrt{-\alpha}b_l) = (\beta(a_j,a_l)+\alpha \beta(b_j,b_l))+\sqrt{-\alpha}(-\beta(a_j,b_l)+\beta(b_j,a_l)).$$

This tells us that when $j \ne l$ that $$\beta(a_j,b_l) = \beta(a_j,a_l) = \beta(b_j,b_l) = 0.$$ 

When $j = l$, we know that $\beta(b_j,b_j) =0= \beta(a_j,a_j)$. Lastly, we see that $\beta(a_j, b_j)= -\beta(b_j,a_j)$. 

Let $$U = (a_1,...,a_{n},b_1,...,b_{n}).$$ Then, it follows that $U^TJU = \left(\begin{smallmatrix}0 & U_1 \\-U_1 & 0\end{smallmatrix}\right)$ where $U_1$ is a diagonal $n \times n$ matrix.

Since $Aa_j = -\sqrt{\alpha}b_j$ and $Ab_j = \frac{\sqrt{\alpha}}{\alpha} a_j$, then we have  $A = \frac{\sqrt{\alpha}}{\alpha}U \left(\begin{smallmatrix}0 & I_{n} \\ -\alpha I_{n} & 0\end{smallmatrix}\right) U^{-1}$.

We have shown what was needed.
\end{proof}

The following is an example of a Type 4 involution where $\sqrt{-\alpha} \not \in k$.

\begin{beisp}
Let $k = \mathbb{F}_5$ and consider $\alpha = 2$. Note that $\sqrt{-\alpha} = \sqrt{3} \not \in k$.

Consider the matrix $$A = \sqrt{2} \left(\begin{array}{cccc}1 & 0 & 2 & 0 \\0 & 1 & 0 & 2 \\3 & 0 & 4 & 0 \\0 & 3 & 0 & 4\end{array}\right).$$

$\Inn_A$ is a Type 4 involution of $\Sp(4,k)$ since $A^2 =-I$ and each entry of $A$ is a $k$-multiple of $\sqrt{2}$. A basis for $E(A,1)$ that matches the conditions of Lemma \ref{Type4ClassNo} is formed by the vectors $$v_1 = \left(\begin{array}{c} 1 \\0 \\4 \\0\end{array}\right)+\sqrt{2}\left(\begin{array}{c} 1 \\0 \\1 \\0\end{array}\right)$$ and $$ v_2=\left(\begin{array}{c}0 \\ 1 \\0 \\4\end{array}\right)+\sqrt{2}\left(\begin{array}{c}0 \\ 1 \\0 \\1\end{array}\right).$$ 

It can be shown that $$v_3 = \left(\begin{array}{c} 1 \\0 \\4 \\0\end{array}\right)-\sqrt{2}\left(\begin{array}{c} 1 \\0 \\1 \\0\end{array}\right)$$ and $$ v_4=\left(\begin{array}{c}0 \\ 1 \\0 \\4\end{array}\right)-\sqrt{2}\left(\begin{array}{c}0 \\ 1 \\0 \\1\end{array}\right)$$   are basis vectors for $E(A,-1)$ that also match the conditions of Lemma \ref{Type4ClassNo}.

Following the notation of the Lemma \ref{Type4ClassNo}, we have $$U = \left(\begin{array}{cccc}1 & 0 & 1 & 0 \\0 & 1 & 0 & 1 \\4 & 0 & 1 & 0 \\0 & 4 & 0 & 1\end{array}\right),$$ where $U^TJU = \left(\begin{smallmatrix}0 & U_1 \\-U_1 & 0\end{smallmatrix}\right)$ for $U_1 = 2I$ and $A = \frac{\sqrt{2}}{2}U\left(\begin{smallmatrix}0 & I \\-2I & 0\end{smallmatrix}\right)U^{-1}$ .
\end{beisp}

We now find conditions on Type 4 involutions where $\sqrt{-\alpha} \not \in k$ that are equivalent to isomorphy.

\begin{theorem}
\label{type4lemNo}
Assume $\sqrt{-\alpha} \not \in k$. Then, if $\Inn_A$ and $\Inn_B$ are both Type 4 involutions of $\Sp(2n,k)$ where the entries of $A$ and $B$ are $k$-multiples of $\sqrt{\alpha}$, then $\Inn_A$ and $\Inn_B$ are isomorphic over $\Sp(2n,k)$.
\end{theorem}

\begin{proof}

By Lemma \ref{Type4ClassNo}, we can choose a matrix $U \in \Gl(n,k)$ such that $$A = \frac{\sqrt{\alpha}}{\alpha} U \left(\begin{array}{cc}0 &  I_{\frac{n}{2}} \\ -\alpha I_{\frac{n}{2}} & 0\end{array}\right) U^{-1}$$ for $$U  = \left(\begin{array}{cccccccccc}a_1 & a_2 & \cdots & a_\frac{n}{2} &b_1 & b_2 & \cdots & b_\frac{n}{2}  \end{array}\right),$$ where the $a_j+\sqrt{-\alpha}b_j$ are a basis for $E(A,i)$, the $a_j-\sqrt{-\alpha}b_j$ are a basis for $E(A,-i)$, and $U^TJU = \left(\begin{smallmatrix}0 & U_1 \\-U_1 & 0 \end{smallmatrix}\right)$ for diagonal $U_1$.

  Consider $\Inn_A$ and $\Inn_B$ as involutions of $\Sp(2n,k[\sqrt{-\alpha}])$. If $k[\sqrt{-\alpha}] = k[\sqrt{\alpha}]$, then these are Type 3 involutions of $\Sp(2n,k[\sqrt{-\alpha}])$, since $A$ and $B$ would have entries in the field, and $i \in k[\sqrt{-\alpha}]$. Otherwise, if $k[\sqrt{-\alpha}] \ne k[\sqrt{\alpha}]$, then these are Type 4 involutions where $ \sqrt{-\alpha} \in k[\sqrt{-\alpha}]$.

Let 
$$X  = (a_1+\sqrt{-\alpha}b_1,...,a_{\frac{n}{2}}+\sqrt{-\alpha}b_{\frac{n}{2}}, a_1-\sqrt{-\alpha}b_1,...,a_{\frac{n}{2}}-\sqrt{-\alpha}b_{\frac{n}{2}}).$$ 

By construction, we see that $X$ is a matrix that satisfies the conditions of Lemma \ref{Type3ClassNo} or Lemma \ref{Type4ClassYes} for the group $\Sp(2n,k[\sqrt{\alpha}])$. We note that $X_1 = -2iU_1$. We also know by Corollary \ref{CorType3} or Theorem \ref{type4lemYes} that $\Inn_A$ and $\Inn_B$ are isomorphic (when viewed as involutions of $\So(n,k[\sqrt{-\alpha}],\beta)$) over $\oo(n,k[\sqrt{-\alpha}],\beta)$. So, we can choose $Q_{\alpha} \in \Sp(2n,k[\sqrt{-\alpha}])$ such that $Q_{\alpha}^{-1}AQ_{\alpha} = B$. Let $Y = Q_{\alpha}^{-1}X$. Since $Y$ is constructed by doing row operations on $X$, then we can write 
$$Y  = (c_1+\sqrt{-\alpha}d_1,...,c_{\frac{n}{2}}+\sqrt{-\alpha}d_{\frac{n}{2}}, c_1-\sqrt{-\alpha}d_1,...,c_{\frac{n}{2}}-\sqrt{-\alpha}c_{\frac{n}{2}}),$$ 
where $c_j, d_j \in k^n$. We now show a couple of facts about $Y$.

First, we note that since $Y$ was obtained from $X$ via row operations, then for $1 \le j \le \frac{n}{2}$, the $j$th and $\frac{n}{2}+j$th columns are $i$-conjugates of one another.

Next, we observe that 
\begin{align*}
Y^{-1}BY &= (Q_{\alpha}^{-1}X)^{-1}B (Q_{\alpha}^{-1}X)\\ 
&= X^{-1}Q_{\alpha}BQ_{\alpha}^{-1}X\\
&= X^{-1}AX\\ 
&= \left(\begin{array}{cc} iI_{\frac{n}{2}} & 0 \\0&  -iI_{\frac{n}{2}} \end{array}\right).
\end{align*}

Lastly, we see that 
\begin{align*}
Y^TJY &= (Q_{\alpha}^{-1}X)^TJ(Q_{\alpha}^{-1}X)\\
&= X^T((Q_{\alpha}^{-1})^TJQ_{\alpha})X\\
&= X^TJX \\
&= \left(\begin{array}{cc} 0 & X_1 \\ -X_1 & 0 \end{array}\right)\\ 
&= \left(\begin{array}{cc} 0 & -2iU_1 \\ 2iU_1 & 0 \end{array}\right).
\end{align*}

Let $$V = (c_1,...,c_{\frac{n}{2}},d_1,...,d_{\frac{n}{2}}) \in \Gl(n,k).$$ It follows from what we have shown that $B = \frac{\sqrt{\alpha}}{\alpha}V \left(\begin{smallmatrix}0 & I_{\frac{n}{2}} \\ -\alpha I_{\frac{n}{2}} & 0\end{smallmatrix}\right) V^{-1}$ where $V^TJV = \left(\begin{smallmatrix}0 & U_1 \\ -U_1 & 0\end{smallmatrix}\right) = U^TJU$.

Now, let $Q = UV^{-1}$. We will show that $Q^{-1}AQ = B$ and $Q \in \Sp(2n,k)$. This will prove that $\Inn_A$ and $\Inn_B$ are isomorphic over $\Sp(2n,k)$. 

We first show that $Q \in \Sp(2n,k)$. 
\begin{align*}
Q^TJQ &= (UV^{-1})^TJUV^{-1}\\
&= (V^{-1})^T(U^TJU)V^{-1}\\
&= (V^{-1})^T(V^TJV)V^{-1}\\ 
&= J.
\end{align*}

Lastly, we show that $Q^{-1}AQ = B$. 
\begin{align*}
Q^{-1}AQ &= (UV^{-1})^{-1}A(UV^{-1})\\ 
&= VU^{-1}AUV^{-1}\\
&=  \frac{\sqrt{\alpha}}{\alpha}V \left(\begin{array}{cc}0 & I_{\frac{n}{2}} \\ -\alpha I_{\frac{n}{2}} & 0\end{array}\right) V^{-1}\\ 
&= B.
\end{align*}
We have shown what was needed.
\end{proof}

Combining the results from this section, we get the following corollary. 

\begin{cor}
\label{CorType4}
If $\Inn_A$ and $\Inn_B$ are both Type 4 involutions of $\Sp(2n,k)$, then $\Inn_A$ and $\Inn_B$ are isomorphic over $\Sp(2n,k)$ if and only if $A$ and $B$ have entries lying in the same field extension of $k$. That is, $\Sp(2n,k)$ has at most $|k^*/(k^*)^2|-1$ isomorphy classes of Type 4 involutions.
\end{cor}

\section{Maximal Number of Isomorphy classes}

From the work we have done, it follows that the maximum number of isomorphy classes of $\Sp(2n,k)$ is a function of the number of square classes of $k$ and $n$. We first define the following formulas.

\begin{definit}

Let $C_1(2n,k)$, $C_2(2n,k)$, $C_3(2n,k)$ and $C_4(2n,k)$ be the number of isomorphy classes of involutions of $\Sp(2n,k)$ of types 1, 2, 3, and 4, respectively.

\end{definit}

From our previous work, we have the following:

\begin{cor}

\begin{enumerate}

\item If $n$ is odd, then $C_1(2n,k) = \frac{n-1}{2}.$ If $n$ is even, then $C_1(2n,k) = \frac{n}{2} .$

\item   If $n$ is odd, then $C_2(2n,k) = 0.$ If $n$ is even, then $C_2(2n,k) \le |k^*/(k^*)^2|-1 .$

\item  $C_3(2n,k) = 1 .$

\item  $C_4(2n,k)  \le |k^*/(k^*)^2|-1 .$ 

\end{enumerate}
\end{cor}

\section{Explicit Examples}

We have shown that the  number of isomorphy classes of Type 1 and Type 3 involutions depends only on $n$, and not the field $k$. Since Type 2 and Type 4 involutions do not occur when $k$ is algebraically closed, then the previous corollary tells us the number of isomorphy classes in this case. In addition to this example, we will also consider the cases where $k = \mathbb{R}$ and $k = \mathbb{F}_p$.

\subsection{Type 2 Examples}

We first consider the Type 2 case. So, we may assume that $n$ is even. First, let us suppose that $k$ is $\mathbb{R}$ or $\mathbb{F}_q$ where $-1$ is not a square in $\mathbb{F}_q$. Without loss of generality, assume $\alpha = -1$. Let $A_1$ be an $n \times n$ block diagonal matrix where each block is the $2 \times 2$ matrix $i\left(\begin{smallmatrix}0 & 1 \\1 & 0\end{smallmatrix}\right)$. Then, let $A = \left(\begin{smallmatrix}A_1 & 0 \\0 & (A_1^{-1})^T\end{smallmatrix}\right).$ This matrix induces a Type 2 involution on $\Sp(2n,k)$. 

Now, let us suppose that $k = \mathbb{F}_q$ where $-1$ is a square. Let $\alpha \in k^*$ be a non-square. Then, we can choose $a,b \in k$ such that $a^2+b^2 = \frac{1}{\alpha}$. Let $A_1$ be an $n \times n$ block diagonal matrix where each block is the $2 \times 2$ matrix $\sqrt{\alpha}\left(\begin{smallmatrix}a & b \\b & -a\end{smallmatrix}\right)$. Then, let $A = \left(\begin{smallmatrix}A_1 & 0 \\0 & (A_1^{-1})^T\end{smallmatrix}\right).$ This matrix induces a Type 2 involution on $\Sp(2n,k)$. So if $k$ is finite or real, then $\Sp(2n,k)$ has the maximal number of Type 2 isomorphy classes.

\subsection{Type 4 Examples}

Now we consider the Type 4 case. So, $n$ may be even or odd. Let us again begin by supposing that $k$ is $\mathbb{R}$ or $\mathbb{F}_q$ where $-1$ is not a square in $\mathbb{F}_q$. Then, the matrix $\left(\begin{smallmatrix}iI_n & 0 \\0 & -iI_n\end{smallmatrix}\right)$ induces a Type 4 involution, and $\Sp(2n,k)$ has the maximal number of isomorphy classes in this case, regardless of if $n$ is odd or even.

Now, let us suppose that $k = \mathbb{F}_q$ where $-1$ is a square.  Let $\alpha \in k^*$ be a non-square and choose $a,b \in k$ such that $a^2+b^2 =\alpha$. If we let $U = \left(\begin{smallmatrix}cI_n & dI_n \\-dI_n & cI_n\end{smallmatrix}\right)$ and then let 
\begin{align*}
A &= \frac{\sqrt{\alpha}}{\alpha}U\left(\begin{array}{cc}0 & I_n \\-\alpha I_n & 0\end{array}\right)U^{-1}\\
&= \frac{\sqrt{\alpha}}{\alpha^2}\left(\begin{array}{cc}(1-\alpha)cdI_n & (c^2+\alpha d^2)I_n \\-(c^2+\alpha d^2)I_n & -(1-\alpha)cdI_n\end{array}\right).
\end{align*}
$A$ induces a Type 4 involution on $\Sp(2n,k)$. We have shown that if $k$ is finite or real, then $\Sp(2n,k)$ has the maximal number of Type 4 isomorphy classes. Thus, if $k$ is real or finite it has the maximal number of all types of isomorphy classes.

While we have been unable to prove that this is the case for any field $k$, we believe that this is the case That is, we have the following conjecture:

\begin{conj}
\begin{enumerate}

\item If $n$ is odd, then $C_1(2n,k) = \frac{n-1}{2}.$ If $n$ is even, then $C_1(2n,k) = \frac{n}{2} .$

\item   If $n$ is odd, then $C_2(2n,k) = 0.$ If $n$ is even, then $C_2(2n,k) = |k^*/(k^*)^2|-1 .$

\item  $C_3(2n,k) = 1 .$

\item  $C_4(2n,k)  = |k^*/(k^*)^2|-1 .$ 

\end{enumerate}
\end{conj}

We have classified the involutions for symplectic groups over algebraically closed fields, the real numbers, and for a finite field of characteristic not 2. We also have constructed the tools to classify the involutions of other symplectic groups. In addition to proving (or disproving) the above conjecture, further areas of research in this area to be completed are to classify the $(\theta, k)$-split tori for given involutions $\theta$, classify the $k$-inner elements, and to study the fixed point groups, which would give rise to a symmetric space.


\begin{thebibliography}{vdBS97}

\bibitem[Art91]{Artin91}
M.~Artin.
\newblock {\em Algebra}.
\newblock Prentice Hall, Englewood Cliffs, NJ, 1991.

\bibitem[Bor91]{Borel91}
A.~Borel.
\newblock {\em Linear Algebraic Groups}, volume 126 of {\em Graduate texts in
  mathematics}.
\newblock Springer Verlag, New York, 2nd enlarged edition edition, 1991.

\bibitem[BT65]{Borel-Tits65}
A.~Borel and J.~Tits.
\newblock Groupes r\'eductifs.
\newblock {\em Inst. Hautes \'Etudes Sci. Publ. Math.}, 27:55--152, 1965.

\bibitem[BT72]{Borel-Tits72}
A.~Borel and J.~Tits.
\newblock Compl\'ements a l'article ``groupes r\'eductifs''.
\newblock {\em Inst. Hautes \'Etudes Sci. Publ. Math.}, 41:253--276, 1972.

\bibitem[Cur84]{Curt84}
M.~Curtis
\newblock{\em Matrix Groups}, Springer-Verlag, New York, Second Edition, 1984.

\bibitem[Hel88]{Helm88}
A.~G. Helminck.
\newblock Algebraic groups with a commuting pair of involutions and semisimple
  symmetric spaces.
\newblock {\em Adv. in Math.}, 71:21--91, 1988.

\bibitem[Helm2000]{Helm2000}
Helminck, A.~G., 
\newblock On the Classification of $k$-involutions {{I}}. 
\newblock {\em Adv. in Math.} (2000).{\bf 153}(1), 1--117. 

\bibitem[HW93]{Helm-Wang93}
A.~G. Helminck and S.~P. Wang.
\newblock On rationality properties of involutions of reductive groups.
\newblock {\em Adv. in Math.}, 99:26--96, 1993.

\bibitem[HW02]{Helm-Wu2002}
Aloysius~G. Helminck and Ling Wu.
\newblock Classification of involutions of ${\rm {S}{L}}(2,k)$.
\newblock {\em Comm. Algebra}, 30(1):193--203, 2002.

\bibitem[HWD04]{HWD04}
Aloysius~G. Helminck, Ling Wu and Christopher Dometrius.
\newblock Involutions of $\Sl(n, k)$, $(n > 2)$.
\newblock {\em Acta Appl. Math.}, {\bf 90}, 91-119, 2006.

\bibitem[Hum75]{Humph75}
J.~E. Humphreys.
\newblock {\em Linear algebraic groups}, volume~21 of {\em Graduate Texts in
  Mathematics}.
\newblock Springer Verlag, New York, 1975.

\bibitem[Ome78]{Omeara}
O.~O'Meara
\newblock{\em Symplectic Groups}, volume 16 of {\em Mathematical Surveys}.
\newblock  American Mathematical Society, Providence, Rhode Island, 1978.

\bibitem[Sch85]{Scharlau}
W.~Scharlau.
\newblock{\em Quadratic and Hermitian Forms},  volume 270 of {\em Grundlehren der mathematischen Wissenschaften}.
\newblock Springer Verlag, Berlin-Heidelberg-New York-Tokyo, 1985.

\bibitem[Spr81]{Spring81}
T.~A. Springer.
\newblock {\em Linear algebraic groups}, volume~9 of {\em Progr. Math.}
\newblock Birkh\"{a}user, Boston/Basel/Stuttgart, 1981.

\bibitem[Szy97]{Szymiczek}
K.~Szymiczek.
\newblock {\em Bilinear Algebra:  An Introduction to the Algebraic Theory of
  Quadratic Forms}, volume 7 of {\em Algebra, Logic and Applications}.
\newblock Gordon and Breach Science Publishers, Amsterdam, 1997.

\bibitem[Wey46]{Weyl}
H.~Weyl.
\newblock {\em The Classical Groups, Their Invariants and Representations},
\newblock Princeton University Press, Princeton, New Jersey, 1946.

\bibitem[Wu02]{Wu2002}
L.~Wu.
\newblock {\em Classification of Involutions of $\Sl(n,k)$
and $\So(2n+1,k)$}.
\newblock PhD Thesis, North Carolina State University, 2002.



\end{thebibliography}
\end{document}